\documentclass[final,leqno]{siamltex}
\pdfoutput=1
\usepackage{amsmath}
\usepackage{amssymb}
\usepackage{graphicx}
\newtheorem{remark}{Remark}[section]
\newtheorem{thm}{Theorem}[section]

\newtheorem{lem}{Lemma}[section]
\newcommand{\lt}{{L^2}}

\newcommand{\htw}{{H^2}}

\newcommand{\eps}{\varepsilon}

\newcommand{\spa}{\hspace{1cm}}

\newcommand{\Tone}{T^{(1)}(I_h \sigma^\varepsilon,I_h u^\varepsilon)}
\newcommand{\Ttwo}{T^{(2)}(I_h \sigma^\varepsilon,I_h u^\varepsilon)}
\newcommand{\vepsi}{{\varepsilon}}
\newcommand{\Ome}{{\Omega}}
\newcommand{\p}{{\partial}}
\newcommand{\Del}{{\Delta}}
\newcommand{\Div}{{\mbox{\rm div}}}

\begin{document}

\title{Mixed finite element methods for the
fully nonlinear Monge-Amp\`ere equation based on the vanishing moment
method\thanks{This work was partially supported by the NSF grants
DMS-0410266 and DMS-0710831.}}

\author{Xiaobing Feng\thanks{Department of Mathematics, The University of
Tennessee, Knoxville, TN 37996, U.S.A. (xfeng@math.utk.edu).}
\and
Michael Neilan\thanks{Department of Mathematics, The University of
Tennessee, Knoxville, TN 37996, U.S.A. (neilan@math.utk.edu).}
}

\maketitle

\setcounter{page}{1}


\begin{abstract}
This paper studies mixed finite element approximations of the viscosity
solution to the Dirichlet problem for the fully nonlinear
Monge-Amp\`ere equation $\det(D^2u^0)=f \,(>0)$ based on the
vanishing moment method which was proposed recently by the authors
in \cite{Feng2}. In this approach, the second order fully nonlinear
Monge-Amp\`ere equation is approximated by the fourth order
quasilinear equation $-\varepsilon\Delta^2 u^\varepsilon +
\det{D^2u^\varepsilon} =f$. It was proved in \cite{Feng1} that the
solution $u^\varepsilon$ converges to the unique convex viscosity
solution $u^0$ of the Dirichlet problem for the Monge-Amp\`ere
equation. This result then opens a door for constructing convergent
finite element methods for the fully nonlinear second order
equations, a task which has been impracticable before. The goal of
this paper is threefold. First, we develop a family of
Hermann-Miyoshi type mixed finite element methods for approximating
the solution $u^\varepsilon$ of the regularized fourth order
problem, which computes simultaneously $u^\vepsi$ and the moment
tensor $\sigma^\vepsi:=D^2u^\varepsilon$. Second, we derive error
estimates, which track explicitly the dependence of the error
constants on the parameter $\vepsi$, for the errors
$u^\varepsilon-u^\varepsilon_h$ and $\sigma^\vepsi-\sigma_h^\vepsi$.
Finally, we present a detailed numerical study on the rates of
convergence in terms of powers of $\vepsi$ for the error
$u^0-u_h^\vepsi$ and $\sigma^\vepsi-\sigma_h^\vepsi$, and
numerically examine what is the ``best" mesh size $h$ in relation to
$\vepsi$ in order to achieve these rates. Due to the strong
nonlinearity of the underlying equation, the standard perturbation
argument for error analysis of finite element approximations of
nonlinear problems does not work for the problem. To overcome the
difficulty, we employ a fixed point technique which strongly relies
on the stability of the linearized problem and its mixed finite
element approximations.
\end{abstract}



\pagestyle{myheadings}
\thispagestyle{plain}
\markboth{XIAOBING FENG AND MICHEAL NEILAN}{MIXED FINITE ELEMENT METHODS FOR
MONGE-AMP\'ERE EQUATIONS}

%
\section{Introduction}\label{sec-1}

This paper is the second in a sequence (cf. \cite{Feng3}) which
concerns with finite element approximations of viscosity solutions
of the following Dirichlet problem for the fully nonlinear
Monge-Amp\`ere equation (cf. \cite{Gutierrez01}):
\begin{alignat}{2}\label{monge1}
\det(D^2 u^0)&=f &&\qquad\text{in }\Omega\subset \mathbf{R}^n,\\
\label{monge2}
u^0&=g &&\qquad\text{on }\partial\Omega,
\end{alignat}
where $\Omega$ is a convex domain with
smooth boundary $\p\Ome$. $D^2 u^0(x)$ and $\det(D^2 u^0(x))$ denote
the Hessian of $u^0$ at $x\in \Ome$ and the determinant of $D^2 u^0(x)$.

The Monge-Amp\`ere equation is a prototype of fully
nonlinear second order PDEs which have a general form
\begin{equation}\label{generalPDE}
F(D^2 u^0, Du^0, u^0, x)=0
\end{equation}
with $F(D^2 u^0, Du^0, u^0, x)= \det(D^2 u^0)-f$.
The Monge-Amp\`ere equation arises naturally from differential geometry
and from applications such as mass transportation, meteorology,
and geostrophic fluid dynamics \cite{Benamou_Brenier00,Caffarelli_Milman99}.
It is well-known that for non-strictly convex domain $\Ome$ the above
problem does not have classical solutions in general even $f$, $g$
and $\p\Ome$ are smooth (see \cite{Gilbarg_Trudinger01}). Classical
result of A. D. Aleksandrov states that the Dirichlet problem
with $f>0$ has a unique generalized solution in the class of
convex functions (cf. \cite{Aleksandrov61,Cheng_Yau77}).
Major progress on analysis of problem \eqref{monge1}-\eqref{monge2}
has been made later after the introduction and establishment
of the viscosity solution theory (cf.
\cite{Caffarelli_Cabre95,Crandall_Ishii_Lions92,Gutierrez01}).
We recall that the notion of viscosity solutions was first introduced by
Crandall and Lions \cite{Crandall_Lions83} in 1983 for the
first order fully nonlinear Hamilton-Jacobi equations. It was
quickly extended to second order fully nonlinear PDEs, with
dramatic consequences in the wake of a breakthrough of Jensen's
maximum principle \cite{Jensen88} and the Ishii's discovery
\cite{Ishii89} that the classical Perron's method
could be used to infer existence of viscosity solutions.
To continue our discussion, we need to recall the definition
of viscosity solutions for the Dirichlet Monge-Amp\`ere
problem \eqref{monge1}-\eqref{monge2} (cf. \cite{Gutierrez01}).
\begin{definition}\label{def0}
a convex function $u^0\in C^0(\overline{\Ome})$ satisfying $u^0=g$
on $\p \Ome$ is called a {\em viscosity subsolution
(resp. viscosity supersolution)} of \eqref{monge1} if for
any $\varphi\in C^2$ there holds $\det(D^2\varphi(x_0))\leq f(x_0)$
(resp. $\det(D^2\varphi(x_0))\geq f(x_0)$) provided
that $u^0-\varphi$ has a local maximum (resp. a local minimum)
at $x_0\in \Ome$. $u^0\in C^0(\overline{\Ome})$ is called a {\em viscosity
solution} if it is both a {\em viscosity subsolution} and
a {\em viscosity supersolution}.
\end{definition}

It is clear that the notion of viscosity solutions is not variational.
It is based on a ``{\em differentiation by parts}" approach,  instead
of the more familiar integration by parts approach. As a result, it is
not possible to directly approximate viscosity solutions
using Galerkin type numerical methods such as finite element,
spectral and discontinuous Galerkin methods,
which all are based on variational formulations of PDEs. The
situation also presents a big challenge and paradox for
the numerical PDE community, since, on one hand, the
``{\em differentiation by parts}" approach has worked remarkably well
for establishing the viscosity solution theory for fully nonlinear
second order PDEs in the past two decades; on the other
hand, it is extremely difficult (if all possible) to mimic this
approach at the discrete level. It should be noted that unlike in the case
of fully nonlinear first order PDEs,  the terminology ``viscosity
solution" loses its original meaning in the case of fully nonlinear
second order PDEs.

Motivated by this difficulty and by the goal of developing
convergent Galerkin type numerical methods for fully nonlinear
second order PDEs, very recently we proposed in \cite{Feng1} a new
notion of weak solutions, called {\em moment solutions}, which is
defined using a constructive method, called the {\em vanishing moment
method}. The main idea of the vanishing moment method is to
approximate a fully nonlinear second order PDE by a quasilinear
higher order PDE. The notion of moment solutions and the vanishing
moment method are natural generalizations of the original definition
of viscosity solutions and the vanishing viscosity method introduced
for the Hamilton-Jacobi equations in \cite{Crandall_Lions83}. We now
briefly recall the definitions of moment solutions and the vanishing
moment method, and refer the reader to \cite{Feng1,Feng2} for a
detailed exposition.

The first step of the vanishing moment method is to approximate the
fully nonlinear equation \eqref{generalPDE} by the following quasilinear
fourth order PDE:
\begin{equation}\label{moment1}
-\vepsi\Delta^2 u^\eps+F(D^2 u^\eps, Du^\eps, u^\eps, x)
=0\spa \ \ (\eps >0),
\end{equation}
which holds in domain $\Ome$. Suppose the Dirichlet boundary condition
$u^0=g$ is prescribed on the boundary $\p\Ome$, then it is natural to impose the
same boundary condition on $u^\vepsi$, that is,
\begin{equation}\label{moment2}
u^\vepsi = g\qquad \text{on }\p\Ome.
\end{equation}
However, boundary condition \eqref{moment2} alone is not sufficient
to ensure uniqueness for fourth order PDEs.  An
additional boundary condition must be imposed.  In \cite{Feng1} the
authors proposed to use one of the following (extra) boundary conditions:
\begin{equation} \label{bc2}
\Delta u^\vepsi=\vepsi, \quad\text{or}\quad
D^2 u^\eps \nu\cdot \nu=\eps\spa \text{on}\ \partial\Omega,
\end{equation}
where $\nu$ stands for the unit outward normal to $\p\Ome$. Although
both boundary conditions work well numerically, the first boundary
condition $\Delta u^\vepsi=\vepsi$ is more convenient for standard
finite element methods, spectral and discontinuous Galerkin methods
(cf. \cite{Feng3}), while the second boundary condition $D^2 u^\eps
\nu\cdot \nu=\eps$ fits better for mixed finite element methods, and
hence, it will be used in this paper.

In summary, the vanishing moment method involves approximating
second order boundary value problem
\eqref{monge2}--\eqref{generalPDE} by fourth order boundary value
problem \eqref{moment1}--\eqref{moment2}, \eqref{bc2}. In the case
of the Monge-Amp\`ere equation, this means that we approximate
boundary value problem \eqref{monge1}--\eqref{monge2} by the
following problem:
\begin{alignat}{2}\label{pdeequation1}
-\vepsi\Delta^2 u^\eps +\text{det}(D^2 u^\eps)&=f \spa &&\text{in}\ \Omega,\\
u^\eps&=g\spa &&\text{on}\ \partial\Omega, \label{pdeequation2} \\
D^2u^\eps\nu\cdot\nu&=\eps\spa &&\text{on}\ \partial\Omega.
\label{pdeequation3}
\end{alignat}
It was proved in \cite{Feng2} that if $f> 0$ in $\Ome$ then
problem \eqref{pdeequation1}--\eqref{pdeequation3} has a unique solution
$u^\vepsi$ which is a strictly convex function over $\Ome$.
Moreover, $u^\vepsi$ uniformly converges as $\vepsi\to 0$ to the
unique viscosity solution of \eqref{monge1}--\eqref{monge2}. As a result,
this shows that \eqref{monge1}--\eqref{monge2} possesses a unique moment
solution that coincides with the unique viscosity solution.
Furthermore, it was proved that there hold the following a priori bounds
which will be used frequently later in this paper:
\begin{alignat}{2}\label{e1.10}
&\|u^\eps\|_{H^j} =O\bigl(\eps^{-\frac{j-1}{2}}\bigr), &&\qquad
 \|u^\eps\|_{W^{2,\infty}}=O\bigl(\eps^{-1}\bigr), \\
&\|D^2u^\eps\|_{L^2} =O\bigl(\eps^{-\frac12}\bigr), &&\qquad
\|\text{cof}(D^2u^\eps)\|_{L^\infty}=O\bigl(\eps^{-1}\bigr) \label{e1.11}
\end{alignat}
for $j=2,3$. Where $\text{cof}(D^2u^\eps)$ denotes the cofactor
matrix of the Hessian, $D^2u^\eps$.

With the help of the vanishing moment methodology, the original
difficult task of computing the unique convex viscosity solution of
the fully nonlinear Monge-Amp\`ere problem
\eqref{monge1}--\eqref{monge2}, which has multiple solutions (i.e.
there are non-convex solutions), is now reduced to a feasible task
of computing the unique regular solution of the quasilinear fourth
order problem \eqref{pdeequation1}--\eqref{pdeequation3}. This then
opens a door to let one use and/or adapt the wealthy amount of
existing numerical methods, in particular, finite element Galerkin
methods to solve the original problem \eqref{monge1}--\eqref{monge2}
via the problem \eqref{pdeequation1}--\eqref{pdeequation3}.

The goal of this paper is to construct and analyze a class of
Hermann-Miyoshi type mixed finite element methods for approximating
the solution of \eqref{pdeequation1}--\eqref{pdeequation3}. In
particular, we are interested in deriving error bounds that exhibit
explicit dependence on $\vepsi$. We note that finite element
approximations of fourth order PDEs, in particular, the biharmonic
equation, were carried out extensively in 1970's in the
two-dimensional case (see \cite{Ciarlet78} and the references
therein), and have attracted renewed interests lately for
generalizing the well-know $2$-D finite elements to the $3$-D case
(cf. \cite{Wang_Shi_Xu07,Wang_Xu07,Tai_Wagner00}) and for developing
discontinuous Galerkin methods in all dimensions (cf.
\cite{Feng_Karakashi07,Suli03}). Clearly, all these methods can be
readily adapted to discretize problem
\eqref{pdeequation1}--\eqref{pdeequation3} although their
convergence analysis do not come easy due to the strong nonlinearity
of the PDE \eqref{pdeequation1}.  
We refer the reader to \cite{Feng3,Neilan_thesis} for further
discussions in this direction.

A few attempts and results on numerical approximations of the
Monge-Amp\`ere as well as related equations have recently been reported in the
literature.  Oliker and Prussner \cite{Oliker_Prussner88}
constructed a finite difference scheme for computing Aleksandrov measure induced
by $D^2u$ in $2$-D and obtained the solution $u$ of problem
\eqref{pdeequation1}--\eqref{pdeequation3} as a by-product.
Baginski and Whitaker \cite{Baginski_Whitaker96}
proposed a finite difference scheme for Gauss curvature equation
(cf. \cite{Feng2} and the references therein) in $2$-D by mimicking the unique
continuation method (used to prove existence of the PDE) at the discrete level.
In a series of papers (cf. \cite{Dean_Glowinski06b} and the references therein)
Dean and Glowinski proposed an augmented Lagrange multiplier method
and a least squares method for problem \eqref{pdeequation1}--\eqref{pdeequation3}
and the Pucci's equation (cf. \cite{Caffarelli_Cabre95,Gilbarg_Trudinger01}) in
$2$-D by treating the Monge-Amp\`ere equation and Pucci's equation as a constraint
and using a variational criterion to select a particular solution.
Very recently, Oberman \cite{Oberman07} constructed some wide stencil finite
difference scheme which fulfill the convergence criterion established by Barles
and Souganidis in \cite{Barles_Souganidis91} for finite difference approximations of
fully nonlinear second order PDEs. Consequently, the convergence
of the proposed wide stencil finite difference scheme immediately follows
from the general convergence framework of \cite{Barles_Souganidis91}.
Numerical experiments results were reported in \cite{Oliker_Prussner88,Oberman07,
Baginski_Whitaker96,Dean_Glowinski06b}, however, convergence analysis was not
addressed except in \cite{Oberman07}.

The remainder of this paper is organized as follows.
In Section \ref{sec-2}, we first derive the Hermann-Miyoshi
mixed weak formulation for problem \eqref{pdeequation1}-\eqref{pdeequation3}
and then present our mixed finite element methods based on this
weak formulation. 
Section \ref{sec-3} is devoted to studying the linearization of
problem \eqref{pdeequation1}-\eqref{pdeequation3} and its mixed
finite element approximations. The results of this section, which
are of independent interests in themselves, will play a crucial role
in our error analysis for the mixed finite element introduced in
Section \ref{sec-2}. In Section \ref{sec-4}, we establish 
error estimates in the energy norm for the proposed mixed
finite element methods. Our main ideas are to use a fixed point
technique and to make strong use of the stability property of the
linearized problem and its finite element
approximations, which all are established in Section \ref{sec-3}. In
addition, we derive the optimal order error estimate in the
$H^1$-norm for $u^\vepsi-u^\vepsi_h$ using a duality argument.
Finally, in Section \ref{sec-5}, we first run some numerical tests
to validate our theoretical error estimate results, we then present
a detailed computational study for determining the ``best'' choice
of mesh size $h$ in terms of $\varepsilon$ in order to achieve the
optimal rates of convergence, and for estimating the rates of
convergence for both $u^0-u^\eps_h$ and $u^0-u^\eps$ in terms of
powers of $\eps$.

We conclude this section by remarking that standard space notations
are adopted in this paper, we refer to
\cite{Brenner,Gilbarg_Trudinger01,Ciarlet78} for their exact
definitions. In addition, $\Ome$ denotes a bounded domain in
$\mathbf{R}^n$ for $n=2,3$. $(\cdot,\cdot)$ and $\langle\cdot,
\cdot\rangle$ denote the $L^2$-inner products on $\Ome$ and on
$\p\Ome$, respectively. For a Banach space $B$, its dual space
is denoted by $B^*$.  $C$ is used to denote a generic
$\vepsi$-independent positive constant.

\section{Formulation of mixed finite element methods}\label{sec-2}

There are several popular mixed formulations for fourth order
problems (cf. \cite{Brezzi_Fortin92,Ciarlet78,Falk_Osborn80}).
However, since the Hessian matrix, $D^2u^\eps$ appears in
\eqref{pdeequation1} in a nonlinear fashion, we cannot use $\Delta
u^\eps$ alone as our additional variables, but rather we are forced
to use $\sigma^\eps:=D^2u^\eps$ as a new variable. Because of
this, we rule out the family of Ciarlet-Raviart mixed finite
elements (cf. \cite{Ciarlet78}). On the other hand, this observation
suggests to try Hermann-Miyoshi or Hermann-Johnson mixed elements
(cf. \cite{Brezzi_Fortin92,Falk_Osborn80}), which both seek
$\sigma^\eps$ as an additional unknown. In this paper, we shall only
focus on developing Hermann-Miyoshi type mixed methods.

We begin with a few more space notation:
\begin{alignat*}{2}
&V:=H^1(\Omega),\spa
&&W:=\{\mu\in [H^1(\Omega)]^{n\times n};\ \mu_{ij}=\mu_{ji}\},\\
&V_0:=H^1_0(\Ome), \spa 
&&V_g:=\{v\in V;\ v|_{\partial\Omega}=g\},\\
&W_\vepsi:=\{\mu\in W;\ \mu n\cdot n |_{\partial\Omega}=\varepsilon\}, \spa
&& W_0:=\{\mu\in W;\ \mu n\cdot n |_{\partial\Omega}=0\}.
\end{alignat*}

To define the Hermann-Miyoshi mixed formulation for problem
\eqref{pdeequation1}-\eqref{pdeequation3}, we rewrite the PDE into
the following system of second order equations:
\begin{eqnarray}\label{equationone}
\sigma^\eps-D^2u^\eps&=&0,\\
-\eps\Delta \text{tr}(\sigma^\eps)+\text{det}(\sigma^\eps)&=&f.
\label{equationtwo}
\end{eqnarray}
Testing \eqref{equationtwo} with $v\in V_0$ yields
\begin{equation}\label{formulation1}
\varepsilon\int_\Omega \text{div}(\sigma^\varepsilon)\cdot D v\, dx
+\int_\Omega \text{det}(\sigma^\eps) v\, dx=\int_\Omega f v\, dx.
\end{equation}
Multiplying \eqref{equationone} by $\mu\in W_0$ and integrating
over $\Omega$ we get
\begin{equation}\label{formulation2}
\int_\Omega \sigma^\eps:\mu\, dx
+\int_\Omega D u^\varepsilon\cdot \text{div}(\mu)\, dx
= \sum_{k=1}^{n-1}\int_{\partial\Omega}
\mu n\cdot\tau_k \frac{\partial g}{\partial \tau_k}\ ds,
\end{equation}
where $\sigma^\eps:\mu$ denotes the matrix inner product and
$\{\tau_1(x), \tau_2(x)\cdots, \tau_{n-1}(x)\}$ denotes the standard
basis for the tangent space to $\p\Ome$ at $x$.

From \eqref{formulation1} and \eqref{formulation2} we define the
variational formulation for \eqref{equationone}-\eqref{equationtwo} as follows:
Find $(u^\varepsilon,\sigma^\varepsilon)\in V_g\times W_\varepsilon$ such that
\begin{align}\label{prob1}
(\sigma^\varepsilon,\mu)+\bigl(\text{div}(\mu),Du^\eps\bigr)
&=\langle \tilde{g},\mu\rangle\spa \forall \mu\in W_0,\\
\bigl(\text{div}(\sigma^\eps),v\bigr)
+\frac{1}{\eps}\bigl(\text{det}\sigma^\eps,v\bigr)&=(f^\eps,v)\spa
\forall v\in V_0, \label{prob2}
\end{align}
where
\[
\langle \tilde{g},\mu\rangle =\sum_{i=1}^{n-1}
\Bigl\langle\frac{\partial g}{\partial \tau_i},\mu n\cdot \tau_i\Bigr\rangle
\quad\text{and}\quad f^\eps=\frac{1}{\eps} f.
\]

\begin{remark}
We note that  $\text{\rm det}(\sigma^\eps)= \frac{1}{n}\Phi^\eps
D^2u^\vepsi =\frac{1}{n}\sum_{i=1}^n
\Phi^\eps_{ij}u^\vepsi_{x_ix_j}$ for $j=1,2,...,n$, where $\Phi^\eps
=\text{\rm cof} (\sigma^\vepsi)$, the cofactor matrix of $\sigma^\vepsi:=D^2u^\eps$.
Thus, using the divergence free property of the cofactor matrix
$\Phi^\eps$ (cf. Lemma \ref{lem3.1}) we can define the following
alternative variational formulation for
\eqref{equationone}-\eqref{equationtwo}:
\begin{align*}
(\sigma^\varepsilon,\mu)+\bigl(\Div(\mu),Du^\eps\bigr)
&=\langle \tilde{g},\mu\rangle\spa \forall \mu\in W_0,\\
(\Div(\sigma^\eps),Dv)-\frac{1}{\eps n}\bigl(\Phi^\eps Du^\eps,Dv\bigr)
&=(f^\eps,v)\spa \forall v\in V_0.
\end{align*}
However, we shall not use the above weak formulation in this paper although
it is interesting to compare mixed finite element methods based on
the above two different but equivalent weak formulations.
\end{remark}

\medskip
To discretize \eqref{prob1}--\eqref{prob2}, let $T_h$ be a quasiuniform
triangular or rectangular partition of $\Ome$ if $n=2$ and be a quasiuniform
tetrahedral or $3$-D rectangular mesh if $n=3$. Let $V^h\subset H^1(\Omega)$
be the Lagrange finite element space consisting of continuous piecewise
polynomials of degree $k(\ge 2)$ associated with the mesh $T_h$. Let
\begin{alignat*}{2}
&V_g^h:=V^h\cap V_g,\spa &&V_0^h:=V^h\cap V_0, \\
&W_\varepsilon^h:=[V^h]^{n\times n}\cap W_\varepsilon,\spa
&&W^h_0:=[V^h]^{n\times n}\cap W_0.
\end{alignat*}

In the $2$-D case, the above choices of $V_0^h$ and $W_0^h$ are known as
the Hermann-Miyoshi mixed finite element for the biharmonic equation
(cf. \cite{Brezzi_Fortin92,Falk_Osborn80}). They form a stable pair
which satisfies the inf-sup condition. We like to note that it is easy to
check that the Hermann-Miyoshi mixed finite element also satisfies the inf-sup
condition in $3$-D. See Section \ref{sec-3.2} for the details.

Based on the weak formulation \eqref{prob1}-\eqref{prob2} and using the
above finite element spaces we now define our Hermann-Miyoshi type
mixed finite element method for \eqref{pdeequation1}--\eqref{pdeequation3}
as follows: Find $(u_h^\varepsilon,\sigma_h^\varepsilon)\in V^h_g\times W^h_\varepsilon$
such that
\begin{align}\label{prob3}
(\sigma^\varepsilon_h,\mu_h)+\bigl(\Div(\mu_h),Du^\eps_h\bigr)
&=\langle \tilde{g},\mu_h\rangle\spa \forall \mu_h\in W^h_0,\\
\bigl(\Div(\sigma^\eps_h),Dv_h\bigr)
+\frac{1}{\eps}\bigl(\det(\sigma^\eps_h),v_h\bigr)
&=(f^\eps,v_h)\spa \forall v_h\in V^h_0. \label{prob4}
\end{align}

Let $(\sigma^\eps,u^\eps)$ be the solution to
\eqref{prob1}-\eqref{prob2} and $(\sigma^\eps_h,u^\eps_h)$ solves
\eqref{prob3}-\eqref{prob4}. As mentioned in Section \ref{sec-1},
the primary goal of this paper is derive error estimates for
$u^\eps-u^\eps_h$ and $\sigma^\eps-\sigma^\eps_h$. To the end, we
first need to prove existence and uniqueness of
$(\sigma^\eps_h,u^\eps_h)$. It turns out both tasks are not easy to
accomplish due to the strong nonlinearity in \eqref{prob4}.  Unlike
in the continuous PDE case where $u^\vepsi$ is proved to be convex
for all $\vepsi$ (cf. \cite{Feng2}), it is far from clear if
$u^\eps_h$ preserves the convexity even for small $\vepsi$ and $h$.
Without a guarantee of convexity for $u^\eps_h$, we could not
establish any stability result for $u^\eps_h$. This in turn makes
proving existence and uniqueness a difficult and delicate task. In
addition, again due to the strong nonlinearity, the standard
perturbation technique for deriving error estimate for numerical
approximations of mildly nonlinear problems does not work here.
To overcome the difficulty, our idea is to adopt a
combined fixed point and linearization technique which was used by
the authors in \cite{Feng_Neilan_Prohl07}, where a nonlinear
singular second order problem known as the inverse mean curvature
flow was studied. We note that this combined fixed point and
linearization technique kills three birds by one stone, that is, it
simultaneously proves existence and uniqueness for $u^\eps_h$ and
also yields the desired error estimates.  In the next two sections,
we shall give the detailed account about the technique and realize
it for problem \eqref{prob3}-\eqref{prob4}.

\section{Linearized problem and its finite element approximations}\label{sec-3}
To build the necessary technical tools, in this section we shall
derive and present a detailed study of the linearization of
\eqref{prob1}-\eqref{prob2} and its mixed finite element approximations,
First, we recall the following divergence-free row property for the
cofactor matrices, which will be frequently used in later sections.
We refer to \cite[p.440]{evans} for a short proof of the lemma.
\begin{lem}\label{lem3.1}
Given a vector-valued function $\mathbf{v}=(v_1,v_2,\cdots,v_n):
\Ome\rightarrow \mathbb{R}^n$. Assume $\mathbf{v}\in [C^2(\Ome)]^n$.
Then the cofactor matrix $\text{\rm cof}(D\mathbf{v})$ of the
gradient matrix $D\mathbf{v}$ of $\mathbf{v}$ satisfies the
following row divergence-free property:
\begin{equation}\label{e3.1}
\Div (\text{\rm cof}(D\mathbf{v}))_i
=\sum_{j=1}^n \p_{x_j} (\text{\rm cof}(D\mathbf{v}))_{ij} =0 \qquad\text{for }
i=1,2,\cdots, n,
\end{equation}
where $(\text{\rm cof}(D\mathbf{v}))_i$ and
$(\text{\rm cof}(D\mathbf{v}))_{ij}$ denote respectively
the $i$th row and the $(i,j)$-entry of $\text{\rm cof}(D\mathbf{v})$.
\end{lem}

\subsection{Derivation of linearized problem}\label{sec-3.1}
We note that for a given function $w$ there holds
\begin{equation*}
\det(D^2u^\eps+tw)=\det(D^2 u^\eps)+t\text{tr}(\Phi^\eps D^2 w)+\cdots
+t^n\text{det}(D^2 w).
\end{equation*}

Thus, setting $t=0$ after differentiating with respect to $t$ we
find the linearization of $M^\eps(u^\vepsi):=-\vepsi\Del^2 u^\vepsi
+\det(D^2u^\vepsi)$ at the solution $u^\eps$ to be
\[
L_{u^\eps}(w):=-\eps \Delta^2 w +\text{tr}(\Phi^\eps D^2 w)
=-\eps \Delta^2 w +\Phi^\eps:D^2 w=-\eps \Delta^2 w+\Div(\Phi^\eps Dw),
\]
where we have used \eqref{e3.1} with $\mathbf{v}=D u^\vepsi$.

We now consider the following linear problem:
\begin{align}\label{linear1}L_{u^\eps}(w)&=q\spa \text{in}\ \Omega,\\
\label{linear2}w&=0\spa \text{on}\ \partial\Omega,\\
\label{linear3} D^2 w\nu \cdot \nu&=0\spa \text{on}\
\partial\Omega.\end{align}

To introduce a mixed formulation for \eqref{linear1}-\eqref{linear3},
we rewrite the PDE as
\begin{align}
\chi-D^2 w&=0,\\
-\eps\Delta \text{tr}(\chi)+\Div(\Phi^\eps Dw)&=q.
\end{align}
Its variational formulation is then defined as:
Given $q\in  V^*_0$, find $(\chi,w)\in W_0\times V_0$ such that
\begin{alignat}{2}\label{linprob1}
(\chi,\mu)+(\Div(\mu),Dw)&=0 &&\spa\forall\mu\in W_0,\\
\label{linprob2} (\Div(\chi),Dv)-\frac{1}{\eps}(\Phi^\eps
Dw,Dv)&=\frac{1}{\eps}\langle q,v\rangle &&\spa\forall v\in V_0.
\end{alignat}

It is not hard to show that if $(\chi,w)$ solves
\eqref{linprob1}-\eqref{linprob2} then $w\in H^2(\Omega)\cap H^1_0(\Ome)$ should
be a weak solution to problem \eqref{linear1}-\eqref{linear3}.
%
On the other hand, by the elliptic theory for linear PDEs (cf. \cite{LU}),
we know that if $q\in V^*_0$, then the solution to problem
\eqref{linear1}-\eqref{linear3} satisfies $w\in H^3(\Omega)$,
so that $\chi=D^2w\in H^1(\Omega)$. It is easy to verify that
$(w,\chi)$ is a solution to \eqref{linprob1}-\eqref{linprob2}.

\subsection{Mixed finite element approximations of the linearized problem}\label{sec-3.2}

Our finite element method for \eqref{linprob1}-\eqref{linprob2} is
defined by seeking $(\chi_h,w_h)\in W^h_0\times V^h_0$ such that
\begin{alignat}{2}\label{linprob1h}
(\chi_h,\mu_h)+(\text{div}(\mu_h),Dw_h)&=0 &&\spa\forall\mu_h\in W^h_0,\\
\label{linprob2h}(\text{div}(\chi_h),Dv_h)-\frac{1}{\eps}(\Phi^\eps
Dw_h,Dv_h)&=\langle q,v_h\rangle &&\spa\forall v_h\in V_0^h.
\end{alignat}

The objectives of this subsection are to first prove existence and
uniqueness for problem \eqref{linprob1h}-\eqref{linprob2h} and
then derive error estimates in various norms. First,
we prove the following inf-sup condition for the
mixed finite element pair $(W^h_0,V^h_0)$.

\begin{lem}
For every $v_h\in V^h_0$, there exists a constant $\beta_0>0$, independent
of $h$, such that
\begin{equation}\label{operatorref2}
\sup_{\ \mu_h\in W^h_0}\ \frac{\bigl(\Div(\mu_h),Dv_h\bigr)}{\|\mu_h\|_{H^1}}
\ge \beta_0 \|v_h\|_{H^1}.
\end{equation}
\end{lem}

\begin{proof}
Given $v_h\in V^h_0$, set $\mu_h=I_{n\times n} v_h$. Then
$\bigl(\Div(\mu_h),Dv_h\bigr)=\|Dv_h\|_\lt^2\ge \beta_0
\|v_h\|_{H^1}^2= \beta_0 \|v_h\|_{H^1}\|\mu_h\|_{H^1}$. Here we have
used Poincare inequality.
\end{proof}

\begin{remark}
By \cite[Proposition 1]{Falk_Osborn80}, \eqref{operatorref2} implies that
there exists a linear
operator $\Pi_h:\ W\to W^h$ such that
\begin{equation}\label{operator1}
\bigl(\Div(\mu-\Pi_h\mu),Dv_h\bigr)=0  \spa\forall v_h\in V^h_0,
\end{equation}
and for $\mu\in W\cap [H^r(\Omega)]^{n\times n},\ r\ge 1$, there holds
\begin{equation}\label{operator2}
\|\mu-\Pi_h\mu\|_{H^j}\le C h^{l-j}\|\mu\|_{H^l}\ \ \ j=0,1,\quad
1\le l\le \min\{k+1,r\}.
\end{equation}
We note that the above results were proved in the $2$-D case in \cite{Falk_Osborn80}, However, they also hold in the $3$-D case as \eqref{operatorref2} holds
in $3$-D.
\end{remark}

\begin{thm}\label{existence}
For any $q\in V^*_0$, there exists a unique solution
$(\chi_h,w_h)\in W_0^h\times V_0^h$ to problem \eqref{linprob1h}-\eqref{linprob2h}.
\end{thm}

\begin{proof}
Since we are in the finite dimensional case and the problem is
linear, it suffices to show uniqueness.  Thus, suppose
$(\chi_h,w_h)\in W_0^h\times V_0^h$ solves
\begin{align*}
(\chi_h,\mu_h)+(\text{div}(\mu_h),Dw_h)&=0\spa\forall\mu_h\in W_0^h,\\
(\text{div}(\chi_h),Dv_h)-\frac1\eps (\Phi^\eps Dw_h,Dv_h)&=0
\spa\forall v_h\in V_0^h.
\end{align*}
Let $\mu_h=\chi_h$, $v_h=w_h$ and subtract two equations to obtain
\begin{equation*}
(\chi_h,\chi_h)+\frac1\eps (\Phi^\eps Dw_h, Dw_h)=0.
\end{equation*}
Since $u^\eps$ is strictly convex, then $\Phi^\eps$ is positive definite.  Thus,
there exists $\theta>0$ such that
\begin{equation*}
\|\chi_h\|_\lt^2+\frac{\theta}{\eps} \|Dw_h\|_\lt^2\le 0.
\end{equation*}
Hence, $\chi_h=0,\ w_h=0$, and the desired result follows.
\end{proof}

\begin{thm}\label{thm0}
Let $(\chi, w)\in [H^r(\Omega)]^{n\times n}\cap W_0\times
H^r(\Omega)\cap V_0$ be the solution to
\eqref{linprob1}-\eqref{linprob2} and $(\chi_h, w_h)\in W^h_0\times
V^h_0$ solves \eqref{linprob1h}--\eqref{linprob2h}. Then there hold
\begin{align}\label{above7}
\|\chi-\chi_h\|_\lt&\le
C\eps^{-\frac32} h^{l-2}\bigl[\|\chi\|_{H^{l}}+\|w\|_{H^l}\bigr]\\
\label{above7a}\|\chi-\chi_h\|_{H^1}&\le
C \eps^{-\frac32} h^{l-3}\bigl[\|\chi\|_{H^{l}}+\|w\|_{H^l}\bigr]\\
\label{above9}\|w-w_h\|_{H^1}&\le
C \eps^{-3} h^{l-1}\bigl[\|\chi\|_{H^{l}}+\|w\|_{H^l}\bigr],\\
\nonumber \text{Moreover, for $k\geq 3$ there also holds}  \\
\label{above10a}\|w-w_h\|_\lt&\le
C\eps^{-5} h^{l} \bigl[\|\chi\|_{H^{l}}+\|w\|_{H^l}\bigr].
\end{align}
\end{thm}

\begin{proof}
Let $I_hw$ denote the standard finite element interplant of $w$ in $V^h_0$.  Then
\begin{align}\label{edit1}
&(\Pi_h\chi-\chi_h,\mu_h)+(\text{div}(\mu_h),D(I_hw-w_h))\\
&\hskip 1in  =(\Pi_h\chi-\chi,\mu_h)+(\Div(\mu_h),D(I_hw-w)),\nonumber\\
\label{edit2}
&(\Div(\Pi_h\chi-\chi_h),Dv_h)-\frac{1}{\eps}\bigl(\Phi^\eps D(I_hw-w_h),Dv_h)\\
&\hskip 1in  =-\frac{1}{\eps}\bigl(\Phi^\eps D(I_h w-w),Dv_h). \nonumber
\end{align}

Let $\mu_h=\Pi_h-\chi_h$ and $v_h=I_hw-w_h$ and subtract \eqref{edit2}
from \eqref{edit1} to get
\begin{align*}
&(\Pi_h\chi-\chi_h,\Pi_h\chi-\chi_h)+\frac{1}{\eps}\bigl(\Phi^\eps
D(I_hw-w_h),D(I_hw-w_h)\bigr)\\
&\hskip 0.7in =(\Pi_h\chi-\chi,\Pi_h\chi-\chi_h)
+\bigl(\Div(\Pi_h\chi-\chi_h),D(I_hw-w)\bigr) \\
&\hskip 1in +\frac{1}{\eps}\bigl(\Phi^\eps D(I_hw-w),D(I_hw-w_h)\bigr).
\end{align*}
Thus,
\begin{align*}
&\|\Pi_h\chi-\chi_h\|_\lt^2+\frac{\theta}{\eps}\|D(I_hw-w_h)\|_\lt^2\\
&\hskip 0.2in \le
\|\Pi_h\chi-\chi\|_\lt\|\Pi_h\chi-\chi_h\|_\lt
+\|\Pi_h\chi-\chi_h\|_{H^1}\|D(I_hw-w)\|_\lt \\
&\hskip 0.5in
+\frac{C}{\eps^2}\|D(I_hw-w)\|_\lt\|D(I_hw-w_h)\|_\lt\\
&\hskip 0.2in \le
\|\Pi_h\chi-\chi\|_\lt\|\Pi_h\chi-\chi_h\|_\lt
+Ch^{-1}\|\Pi_h\chi-\chi_h\|_\lt\|D(I_hw-w)\|_\lt \\
&\hskip 0.5in
+\frac{C}{\eps^2}\|D(I_hw-w)\|_\lt\|D(I_hw-w_h)\|_\lt,
\end{align*}
where we have used the inverse inequality.

Using the Schwarz inequality and rearranging terms yield
\begin{align}\label{above6}
\|\Pi_h\chi &-\chi_h\|_\lt^2 +\frac{1}{\eps}\|D(I_hw-w_h)\|_\lt^2 \\
&\le C\bigl(\|\Pi_h\chi-\chi\|_\lt^2+h^{-2}\|I_hw-w\|_{H^1}^2
+\eps^{-3}\|I_hw-w\|_{H^1}^2\bigr).\nonumber
\end{align}
Hence, by the standard interpolation results \cite{Brenner,Ciarlet78}
we have
\begin{align*}
\|\Pi_h\chi-\chi_h\|_\lt&\le
C\bigl(\|\Pi_h\chi-\chi\|_\lt+h^{-1}\|I_hw-w\|_{H^1}
+\eps^{-\frac32}\|I_hw-w\|_{H^1}\bigr)\\
&\le C\eps^{-\frac32} h^{l-2} \bigl(\|\chi\|_{H^l}+\|w\|_{H^l}\bigr),
\end{align*}
which and the triangle inequality yield
\begin{equation*}
\|\chi-\chi_h\|_\lt\le C \eps^{-\frac32} h^{l-2} \bigl(\|\chi\|_{H^{l}}
+\|w\|_{H^l}\bigr).
\end{equation*}
The above estimate and the inverse inequality yield
\begin{align*}
\|\chi-\chi_h\|_{H^1}&\le \|\chi-\Pi_h\chi\|_{H^1}+\|\Pi_h\chi-\chi_h\|_{H^1}\\
&\le \|\chi-\Pi_h\chi\|_{H^1}+h^{-1}\|\Pi_h\chi-\chi_h\|_\lt\\
&\le C \eps^{-\frac32} h^{l-3} \bigl(\|\chi\|_{H^{l}}+\|w\|_{H^l}\bigr).
\end{align*}

Next, from \eqref{above6} we have
\begin{align}\nonumber
\|D(I_hw-w_h)\|_\lt&\le \sqrt{\eps}
C\bigl[\|\Pi_h\chi-\chi\|_\lt+ h^{-1}\|D(I_hw-w)\|_\lt
+\eps^{-\frac32}\|I_hw-w\|_{H^1}\bigr]\\
\label{above8}
&\le C\eps^{-1} h^{l-2} \bigl(\|\chi\|_{H^{l}}+\|w\|_{H^l}\bigr).
\end{align}

To derive \eqref{above9}, we consider the following auxiliary problem: Find
$z\in H^2(\Ome)\cap H^1_0(\Omega)$ such that
\begin{alignat*}{2}
-\eps \Delta^2 z+\Div(\Phi^\eps Dz)&=-\Delta(w-w_h) &&\qquad\text{in } \Omega,\\
D^2 z\nu\cdot\nu &=0 &&\qquad\text{on } \p\Omega. 
\end{alignat*}
By the elliptic theory for linear PDEs (cf. \cite{LU}), we know that
the above problem has a unique solution $z\in H^1_0(\Ome)\cap H^3(\Ome)$ and
\begin{equation}\label{estimatelin}
\|z\|_{H^3}\le C_b(\eps)\|D(w-w_h)\|_\lt \quad\text{where}\quad
C_b(\eps)=O(\vepsi^{-1}).
\end{equation}
Setting $\kappa=D^2z$, it is easy to verify that $(\kappa,z)\in W_0\times V_0$ and
\begin{alignat*}{2}
(\kappa,\mu)+\bigl(\Div(\mu),Dz\bigr)&=0 &&\spa\forall \mu\in W_0,\\
\bigl(\Div(\kappa),Dv\bigr)-\frac{1}{\eps}\bigl(\Phi^\eps Dz,
Dv\bigr) &=\frac{1}{\eps}(D(w-w_h),Dv) &&\spa \forall v\in V_0.
\end{alignat*}
It is easy to check that \eqref{linprob1h}--\eqref{linprob2h}
produce the following error equations:
\begin{alignat}{2}
(\chi-\chi_h,\mu_h)+(\Div(\mu_h),D(w-w_h))&=0\spa
&&\forall \mu_h\in W^h_0,\\
(\Div(\chi-\chi_h),Dv_h)-\frac{1}{\eps}(\Phi^\eps D(w-w_h),Dv_h)&=0\spa
&&\forall v_h\in V^h_0.
\end{alignat}
Thus,
\begin{align*}
&\hspace{-0.4in}\frac{1}{\eps}\|D(w -w_h)\|_\lt^2=\bigl(\Div(\kappa),D(w-w_h)\bigr)
-\frac{1}{\eps}\bigl(\Phi^\eps Dz,D(w-w_h)\bigr)\\
&=\bigl(\Div(\kappa-\Pi_h\kappa),D(w-w_h)\bigr)-\frac{1}{\eps}\bigl(\Phi^\eps
D z,D(w-w_h)\bigr)\\
&\hspace{0.2in}+\bigl(\Div(\Pi_h \kappa),D(w-w_h)\bigr)\\
&=\bigl(\Div(\kappa-\Pi_h\kappa),D(w-I_hw)\bigr)-\frac{1}{\eps}\bigl(\Phi^\eps
Dz,D(w-w_h)\bigr)\\
&\hspace{0.2in}+\bigl(\chi_h-\chi,\Pi_h\kappa\bigr)\\
&=\bigl(\Div(\kappa-\Pi_h\kappa),D(w-I_hw)\bigr)-\frac{1}{\eps}\bigl(\Phi^\eps
Dz,D(w-w_h)\bigr)\\
&\hspace{0.2in}+\bigl(\chi_h-\chi,\Pi_h\kappa-\kappa\bigr)+\bigl(\chi_h-\chi,\kappa\bigr)\\
&=\bigl(\Div(\kappa-\Pi_h\kappa),D(w-I_hw)\bigr)-\frac{1}{\eps}\bigl(\Phi^\eps
Dz,D(w-w_h)\bigr)\\
&\hspace{0.2in}+\bigl(\chi_h-\chi,\Pi_h\kappa-\kappa\bigr)+\bigl(\Div(\chi-\chi_h),Dz\bigr)\\
&=\bigl(\Div(\kappa-\Pi_h\kappa),D(w-I_hw)\bigr)+(\chi_h-\chi,\Pi_h\kappa-\kappa)\\
&\hspace{0.2in}+(\Div(\chi-\chi_h),D(z-I_hz)\bigr)-\frac{1}{\eps}\bigr(\Phi^\eps D(w-w_h),D(z-I_hz)\bigr)\\
&\le
\|\Div(\kappa-\Pi_h\kappa)\|_\lt\|D(w-I_hw)\|_\lt+\|\chi_h-\chi\|_\lt\|\Pi_h\kappa-\kappa\|_\lt\\
&\hspace{0.2in}+\|\Div(\chi-\chi_h)\|_\lt\|D(z-I_hz)\|_\lt \\
&\hspace{0.2in} +\frac{C}{\eps^2}\|D(z-I_hz)\|_\lt\|D(w-w_h)\|_\lt\\
&\le C\Bigl[\|D(w-I_hw)\|_\lt +
h\|\chi_h-\chi\|_\lt+h^2\|\Div(\chi-\chi_h)\|_\lt\\
&\hspace{0.2in}+\frac{h^2}{\eps^2}\|D(w-w_h)\|_\lt\Bigr]\|z\|_{H^3}.
\end{align*}
Then, by \eqref{above7},\eqref{above7a},\eqref{above8}, and
\eqref{estimatelin}, we have
\begin{equation*}
\|D(w-w_h)\|_\lt\le C_b(\eps) \eps^{-2} h^{l-1}\bigl[\|\chi\|_{H^{l}}
+\|w\|_{H^l}\bigr].
\end{equation*}
Substituting $C_b(\eps)=O(\eps^{-1})$ we get \eqref{above9}.

To derive the $L^2$-norm estimate for $w-w_h$, we consider the following
auxiliary problem: Find $(\kappa,z)\in W_0\times V_0$ such that
\begin{alignat*}{2}
(\kappa,\mu)+\bigl(\Div(\mu),Dz\bigr)&=0 &&\spa\forall \mu\in W_0,\\
\bigl(\Div(\kappa),Dv\bigr)-\frac{1}{\eps}\bigl(\Phi^\eps Dz,
Dv\bigr) &=\frac{1}{\eps}(w-w_h,v) &&\spa \forall v\in V_0.
\end{alignat*}
Assume the above problem is $H^4$-regular, that is, $z\in H^4(\Ome)$ and
\begin{equation}
\label{estimatelin2}\|z\|_{H^4}\le C_b(\eps)\|w-w_h\|_\lt\quad\text{with}\quad
C_b(\vepsi)=O(\vepsi^{-1}).
\end{equation}
We then have
\begin{align*}
\frac{1}{\eps}\|w &-w_h\|_\lt^2=\bigl(\Div(\kappa),D(w-w_h)\bigr)
-\frac{1}{\eps}\bigl(\Phi^\eps D(w-w_h),Dz\bigr)\\
&=\bigl(\Div(\Pi_h\kappa),D(w-w_h)\bigr)
-\frac{1}{\eps}\bigl(\Phi^\eps D(w-w_h),Dz\bigr) \\
&\hspace{0.4in}
+\bigl(\Div(\kappa-\Pi_h\kappa),D(w-w_h\bigr)\\
&=(\chi_h-\chi,\Pi_h\kappa)-\frac{1}{\eps}(\Phi^\eps Dz, D(w-w_h)) \\
&\hspace{0.4in}+(\Div(\kappa-\Pi_h\kappa),D(w-I_hw))\\
&=(\chi_h-\chi,\kappa)+(\chi_h-\chi,\Pi_h\kappa-\kappa)\\
&\hspace{0.4in}
-\frac{1}{\eps}(\Phi^\eps Dz,D(w-w_h))+(\Div(\kappa-\Pi_h\kappa),D(w-I_hw))\\
&=(\Div(\chi-\chi_h),Dz)-\frac{1}{\eps}(\Phi^\eps D(w-w_h),Dz)\\
&\hspace{0.4in}+(\chi_h-\chi,\Pi_h\kappa-\kappa)+(\Div(\kappa-\Pi_h\kappa),D(w-I_hw))\\
&=(\Div(\chi-\chi_h),D(z-I_hz))-\frac{1}{\eps}(\Phi^\eps D(w-w_h),D(z-I_hz))\\
&\hspace{0.4in}+(\chi_h-\chi,\Pi_h\kappa-\kappa)+(\Div(\kappa-\Pi_h\kappa),D(w-I_hw))\\
&\le \bigl[\|\Div(\chi-\chi_h)\|_\lt+\frac{C}{\eps^2} \|D(w-w_h)\|_\lt \bigr]
\|D(z-I_hz)\|_\lt\\
&\hspace{0.2in}+\|\chi_h-\chi\|_\lt\|\Pi_h\kappa-\kappa\|_\lt+\|\Div(\kappa-\Pi_h\kappa)\|_\lt\|D(w-I_hw)\|_\lt\\
&\le
Ch^3\bigl[\|\chi-\chi_h\|_{H^1}+\frac{1}{\eps^2}\|w-w_h\|_{H^1}\bigr]\|z\|_{H^4}\\
&\hspace{0.4in}+Ch^2\|\chi_h-\chi\|_\lt\|\kappa\|_{H^2}
+Ch\|w-I_hw\|_{H^1} \|\kappa\|_{H^2} \\
&\le C\eps^{-5} h^l \bigl(\|\chi\|_{H^l}+\|w\|_{H^l}\bigr)\|z\|_{H^4}\\
&\le C C_b(\eps) \eps^{-5} h^l \bigl(\|\chi\|_{H^l} +\|w\|_{H^l}\bigr)\|w-w_h\|_\lt,
\end{align*}
where we have used \eqref{above7},\eqref{above7a},\eqref{above9},
\eqref{estimatelin2}, and the assumption $k\geq 3$.
Dividing the above inequality
by $\|w-w_h\|_\lt$ and substituting $C_b(\eps)=O(\eps^{-1})$
we get \eqref{above10a}.  The proof is complete.
\end{proof}

\section{Error analysis for finite element method \eqref{prob3}-\eqref{prob4}}
 \label{sec-4}
The goal of this section is to derive error estimates for the finite
element method \eqref{prob3}-\eqref{prob4}. Our main idea is to use
a combined fixed point and linearization technique which was used by
the authors in \cite{Feng_Neilan_Prohl07}.

\begin{definition}
Let $T: W^h_\eps\times V^h_g\to W^h_\eps\times V^h_g$ be a linear
mapping such that for any $(\mu_h,v_h)\in W^h_\eps \times V^h_g$,
$T(\mu_h,v_h)=(T^{(1)}(\mu_h,v_h),T^{(2)}(\mu_h,v_h))$ satisfies
\begin{align}\label{map1}
&\bigl(\mu_h-T^{(1)}(\mu_h,v_h),\kappa_h\bigr)+\bigl(\Div(\kappa_h),
D( v_h-T^{(2)}(\mu_h,v_h))\bigr)\\
&\spa\spa= (\mu_h,\kappa_h)+\bigl(\Div(\kappa_h),Dv_h\bigr)
-\langle \tilde{g},\kappa_h\rangle
\quad \forall \kappa_h\in W^h_0, \nonumber\\
&\bigl(\Div(\mu_h-T^{(1)}(\mu_h,v_h)),Dz_h\bigr)
-\frac{1}{\eps}\bigl(\Phi^\eps D(v_h-T^{(2)}(\mu_h,v_h)),Dz_h\bigr)
\label{map2}\\
&\spa\spa= \bigl(\Div(\mu_h),Dz_h \bigr)
+\frac{1}{\eps}\bigl(\det(\mu_h),z_h\bigr) -(f^\eps,z_h)
\quad \forall z_h\in V_0. \nonumber
\end{align}
\end{definition}

By Theorem \ref{existence},  we conclude that $T(\mu_h,v_h)$ is well
defined. Clearly, any fixed point $(\chi_h,w_h)$ of the mapping $T$
(i.e., $T(\chi_h,w_h)=(\chi_h,w_h)$) is a solution to problem
\eqref{prob3}-\eqref{prob4}, and vice-versa. The rest of this
section shows that indeed the mapping $T$ has a unique fixed point
in a small neighborhood of $(I_h\sigma^\eps,I_h u^\eps)$. To this
end, we define
\begin{align*}
\tilde{B}_h(\rho)&:=\{(\mu_h,v_h)\in W^h_\eps\times V^h_g;\
\|\mu_h-I_h\sigma^\eps\|_\lt
+ \frac{1}{\sqrt{\eps}}\|v_h-I_hu^\eps\|_{H^1}\le \rho\}. \\
\tilde{Z}_h&:=\{(\mu_h,v_h)\in W^h_\eps\times V^h_g;\
(\mu_h,\kappa_h)+(\Div(\kappa_h),Dv_h)=\langle
\tilde{g},\kappa_h\rangle\,\, \forall \kappa_h\in W^h_0\}. \\
B_h(\rho)&:=\tilde{B}_h(\rho)\cap \tilde{Z}_h.
\end{align*}
We also assume $\sigma^\eps\in H^r(\Omega)$ and set $l=\text{min}\{k+1,r\}$.

The next lemma measures the distance between the center of
$B_h(\rho)$ and its image under the mapping $T$.


\begin{lem}\label{lembound5}
The mapping $T$ satisfies the following estimates:
\begin{eqnarray}
&&\|I_h \sigma^\eps-\Tone\|_{H^1} \le
C_1(\eps)h^{l-3}\bigl[\|\sigma^\eps\|_{H^{l}}
+\|u^\eps\|_{H^l}\bigr], \label{eqtwo}\\
&&\|I_h \sigma^\eps-\Tone\|_\lt \le
C_2(\eps)h^{l-2}\bigl[\|\sigma^\eps\|_{H^{l}}
+\|u^\eps\|_{H^l}\bigr], \label{eqthree}\\
\label{eqone} &&\|I_h u^\eps-\Ttwo\|_{H^1} \le
C_3(\eps)h^{l-1}\bigl[\|\sigma^\eps\|_{H^l}+\|u^\eps\|_{H^l}\bigr].
\end{eqnarray}
\end{lem}

\begin{proof}
We divide the proof into four steps.

{\em Step 1}: To ease notation we set
$\omega_h=I_h\sigma^\eps-\Tone$, $s_h=I_hu^\eps-\Ttwo$.  By the
definition of $T$ we have for any $(\mu_h,v_h)\in W^h_0\times V^h_0$
\begin{align*}
&(\omega_h, \mu_h)+\bigl(\Div(\mu_h),Ds_h\bigr) =
(I_h\sigma^\eps,\mu_h)+\bigl(\Div(\mu_h),D(I_hu^\eps)\bigr)
-\langle \tilde{g},\mu_h\rangle,\\
&\bigl(\Div(\omega_h),Dv_h\bigr)-\frac{1}{\eps}\bigl(\Phi^\eps
Ds_h,Dv_h\bigr)=\bigl(\Div(I_h\sigma^\eps),Dv_h\bigr)
+\frac{1}{\eps} \bigl(\det(I_h\sigma^\eps),v_h\bigr) -(f^\eps,v_h).
\end{align*}
It follows from \eqref{prob1}--\eqref{prob2} that for any
$(\mu_h,v_h)\in W^h_0\times V^h_0$
\begin{align}\label{one}
&(\omega_h, \mu_h)+(\text{div}(\mu_h),Ds_h)
=(I_h\sigma^\eps-\sigma^\eps,\mu_h)
+\bigl(\Div(\mu_h),D(I_hu^\eps-u^\eps)\bigr),\\
&\bigl(\Div(\omega_h),Dv_h\bigr)-\frac{1}{\eps}\bigl(\Phi^\eps
Ds_h,Dv_h\bigr)=\bigl(\Div(I_h\sigma^\eps-\sigma^\eps),Dv_h\bigr)
\label{two}\\
&\hskip 2.4in +\frac{1}{\eps}
\bigl(\det(I_h\sigma^\eps)-\det(\sigma^\eps),v_h\bigr). \nonumber
\end{align}

Letting $v_h=s_h$, $\mu_h=\omega_h$ in \eqref{one}-\eqref{two},
subtracting the two equations and using the Mean Value Theorem we
get
\begin{align*}
(\omega_h,\omega_h)+\frac{1}{\eps}\bigl(\Phi^\eps Ds_h, Ds_h\bigr)
&=(I_h\sigma^\eps-\sigma^\eps, \omega_h)
+\bigl(\Div(\omega_h),D(I_hu^\eps-u^\eps)\bigr)\\
&\hspace{0.1in} +\bigl(\Div(\sigma-I_h\sigma^\eps),Ds_h\bigr)
+\frac{1}{\eps}\bigl(\det(\sigma^\eps)
-\text{det}(I_h\sigma^\eps),s_h\bigr)\\
&=(I_h\sigma^\eps-\sigma^\eps,\omega_h\bigr)
+\bigl(\Div(\omega_h),D(I_hu^\eps-u^\eps)\bigr)\\
&\hspace{0.1in}+\bigl(\Div(\sigma-I_h\sigma^\eps),Ds_h\bigr)
+\frac{1}{\eps}\bigl(\Psi^\eps:(\sigma^\eps-I_h\sigma^\eps),s_h\bigr),
\end{align*}
where $\Psi^\eps=\text{cof}(\tau I_h\sigma^\eps
+[1-\tau]\sigma^\eps)$ for $\tau\in [0,1].$

{\em Step 2: The case $n=2$}. Since $\Psi^\eps$ is a $2\times 2$
matrix whose entries are same as those of $\tau I_h\sigma^\eps
+[1-\tau]\sigma^\eps$,  then by \eqref{e1.11} we have
\begin{align*}
\|\Psi^\eps\|_\lt &=\|\text{cof}(\tau
I_h\sigma^\eps+[1-\tau]\sigma^\eps)\|_\lt
=\|\tau I_h\sigma^\eps+[1-\tau]\sigma^\eps\|_\lt\\
&\le \|I_h\sigma^\eps\|_\lt + \|\sigma^\eps\|_\lt \le
C\|\sigma^\eps\|_\lt=O(\eps^{-\frac12}).
\end{align*}

{\em Step 3: The case $n=3$}. Note that $(\Psi^\eps)_{ij}
=(\text{cof}(\tau
I_h\sigma^\eps+[1-\tau]\sigma^\eps))_{ij}=\text{det}(\tau
I_h\sigma^\eps |_{ij}+[1-\tau]\sigma^\eps |_{ij})$, where
$\sigma^\eps |_{ij}$ denotes the $2\times 2$ matrix after deleting
the $i$th row and $j$th column of $\sigma^\eps$. We can thus
conclude that
\begin{align*}
|(\Psi^\eps)_{ij}|&\le 2\max_{s\neq i,t\neq j} \bigl(
|\tau (I_h\sigma^\eps)_{st}+[1-\tau](\sigma^\eps)_{st}|\bigr)^2\\
&\le C \max_{s\neq i, t\neq j} |(\sigma^\eps)_{st}|^2 \le C
\|\sigma^\eps\|_{L^\infty}^2.
\end{align*}
Thus, \eqref{e1.11} implies that
\[
\|\Psi^\eps\|_\lt \le C \|\sigma^\eps\|_{L^\infty}^2=O(\eps^{-2}).
\]

{\em Step 4:} Using the estimates of $\|\Psi^\eps\|_{L^2}$ we have
\begin{align*}
\|\omega_h\|_\lt^2+\frac{\theta}{\eps}\|Ds_h\|_\lt^2 &\le
\|I_h\sigma^\eps-\sigma^\eps\|_\lt \|\omega_h\|_\lt
+\|\omega_h\|_{H^1}\|D(I_hu^\eps-u^\eps)\|_\lt\\
&\quad + \|I_h\sigma^\eps-\sigma^\eps\|_{H^1}\|Ds_h\|_\lt + C(\eps)
\|\sigma^\eps-I_h\sigma^\eps\|_{H^1}\|s_h\|_{H^1},
\end{align*}
where we have used Sobolev inequality. It follows from Poincare inequality,
Schwarz inequality, and the inverse inequality  that
\begin{align}\label{eqabove}
\|\omega_h\|_\lt^2+\frac{\theta}{\eps}\|s_h\|_{H^1}^2 &\le
C(\eps)\|I_h\sigma^\eps-\sigma^\eps\|_{H^1}^2
+C\|\omega_h\|_{H^1}\|I_hu^\eps-u^\eps\|_{H^1}\\
&\nonumber \le C(\eps)h^{2l-2}\|\sigma^\eps\|_{H^l}^2
+Ch^{-1}\|\omega_h\|_\lt\|I_hu^\eps-u^\eps\|_{H^1}.
\end{align}
Hence,
\begin{align*}
\|\omega_h\|_\lt^2+\frac{1}{\eps}\|s_h\|_{H^1}^2
&\le C(\eps)h^{2l-2}\|\sigma^\eps\|_{H^l}^2+Ch^{2l-4}\|u^\eps\|_{H^l}^2.
\end{align*}
Therefore,
\begin{equation*}
\|\omega_h\|_\lt\le C_2(\eps)h^{l-2}\bigl[\|\sigma^\eps\|_{H^l}
+\|u^\eps\|_{H^l}\bigr],
\end{equation*}
which and the inverse inequality yield
\begin{equation*}
\|\omega_h\|_{H^1}\le C_1(\eps)h^{l-3}\bigl[\|\sigma^\eps\|_{H^l}
+\|u^\eps\|_{H^l}\bigr].
\end{equation*}

Next, from \eqref{one} we have
\begin{align*}
(\Div(\mu_h),Ds_h)
&\le \|\omega_h\|_\lt\|\mu_h\|_\lt+\|I_h\sigma^\eps-\sigma^\eps\|_\lt\|\mu_h\|_\lt \\
&\qquad  +\|\Div(\mu_h)\|_\lt\|D(I_hu^\eps-u^\eps)\|_\lt\\
&\le C_2(\eps)h^{l-2}\bigl[\|\sigma^\eps\|_{H^l}
+\|u^\eps\|_{H^l}\bigr]\|\mu_h\|_{H^1}.
\end{align*}
It follows from \eqref{operatorref2} that
\begin{equation}
\label{shh1}\|Ds_h\|_\lt\le C(\eps)h^{l-2}\bigl[\|\sigma^\eps\|_{H^l}
+\|u^\eps\|_{H^l}\bigr].
\end{equation}

To prove \eqref{eqone}, let $(\kappa,z)$ be the solution to
\begin{alignat*}{2}
(\kappa,\mu)+(\Div(\mu),Dz)&=0\spa &&\forall \mu\in W_0,\\
(\Div(\kappa),Dv)-\frac{1}{\eps}(\Phi^\eps Dz,Dv)&=\frac{1}{\eps}(Ds_h,Dv)\spa
&&\forall v\in V_0,
\end{alignat*}
and satisfy
\[
\|z\|_{H^3}\le C_b(\eps)\|Ds_h\|_\lt.
\]
Then,
\begin{align*}
\frac{1}{\eps}\|Ds_h\|_\lt^2&=(\Div (\kappa),Ds_h)
-\frac{1}{\eps}(\Phi^\eps Dz,Ds_h)\\
&=(\Div (\Pi_h\kappa),Ds_h)-\frac{1}{\eps}(\Phi^\eps Dz,Ds_h)\\
&=-(\omega_h,\Pi_h\kappa)-\frac{1}{\eps}(\Phi^\eps
Dz,Ds_h)+(I_h\sigma^\eps-\sigma^\eps,\Pi_h\kappa)\\
&\hspace{0.2in}  +(\Div(\Pi_h\kappa),D(I_hu^\eps-u^\eps))\\
&=-(\omega_h,\kappa)+(\omega_h,\kappa-\Pi_h\kappa)
-\frac{1}{\eps}(\Phi^\eps Dz,Ds_h)\\
&\hspace{0.2in}+(I_h\sigma^\eps-\sigma^\eps,\Pi_h\kappa)
+(\Div(\Pi_h\kappa),D(I_hu^\eps-u^\eps))\\
&=(\Div(\omega_h),Dz)-\frac{1}{\eps}(\Phi^\eps
Ds_h,Dz)+(\omega_h,\kappa-\Pi_h\kappa)\\
&\hspace{0.2in}+(I_h\sigma^\eps-\sigma^\eps,\Pi_h\kappa)
+(\Div(\Pi_h\kappa),D(I_hu^\eps-u^\eps))\\
&=(\Div(\omega_h),D(z-I_hz))-\frac{1}{\eps}(\Phi^\eps Ds_h,D(z-I_hz))
+(\omega_h,\kappa-\Pi_h\kappa)\\
&\hspace{0.2in}+(I_h\sigma^\eps-\sigma^\eps,\Pi_h\kappa)
+(\Div(\Pi_h\kappa),D(I_hu^\eps-u^\eps))\\
&\hspace{0.2in}+(\Div(\sigma^\eps-I_h\sigma^\eps),I_hz)
+\frac{1}{\eps}(\text{det}(\sigma^\eps)-\text{det}(I_h\sigma^\eps),I_hz) \\
&\le \|\Div(\omega_h)\|_\lt\|D(z-I_hz)\|_\lt+\frac{1}{\eps}\|\Phi^\eps\|_{L^\infty}
\|Ds_h\|_\lt\|D(z-I_hz)\|_\lt \\
&\hspace{0.2in} +\|\omega_h\|_\lt\|\kappa-\Pi_h\kappa\|_\lt
+\|I_h\sigma^\eps-\sigma^\eps\|_\lt\|\Pi_h\kappa\|_\lt \\
&\hspace{0.2in} +\|\Div(\Pi_h\kappa)\|_\lt\|D(I_hu^\eps-u^\eps)\|_\lt\\
&\hspace{0.2in}+\|\Div(\sigma^\eps-I_h\sigma^\eps)\|_\lt\|I_hz\|_\lt
+\frac{C}{\eps}\|\Psi^\eps\|_\lt\|\sigma^\eps-I_h\sigma^\eps\|_{H^1}\|I_hz\|_{H^1}
\\
&\le Ch^2\bigl(\|\omega\|_{H^1}+\frac{1}{\eps^2}\|Ds_h\|_\lt\bigr)\|z\|_{H^3}
+C(\eps)h^{l-1}\bigl(\|I_hz\|_\lt+\|I_hz\|_{H^1}\bigr)\|\sigma^\eps\|_{H^l}\\
&\hspace{0.2in}+Ch\|\omega_h\|_\lt\|\kappa\|_{H^1}
+Ch^l\|\sigma^\eps\|_{H^l}\|\Pi_h\kappa\|_\lt
+Ch^{l-1}\|\Pi_h\kappa\|_{H^1}\|u^\eps\|_{H^l}
\end{align*}
\begin{align*}
&\le C_2(\eps)\eps^{-2} h^{l-1} \bigl[\|u^\eps\|_{H^l}
+\|\sigma^\eps\|_{H^l}\bigr]\|z\|_{H^3}\\
&\le C_2(\eps) \eps^{-2} C_b(\eps)h^{l-1} \bigl[\|u^\eps\|_{H^l}
+\|\sigma^\eps\|_{H^l}\bigr]\|Ds_h\|_\lt.
\end{align*}
Dividing by $\|Ds_h\|_\lt$, we get \eqref{eqone}.  The proof is complete.
\end{proof}

\begin{remark}
Tracing the dependence of all constants on $\vepsi$, we find that
$C_1(\eps)=O(1)$, $C_2(\eps)=O(1)$,
$C_3(\eps)=O(\eps^{-2})$ when $n=2$, and
$C_1(\eps)=O(\eps^{-\frac32})$,
$C_2(\eps)=O(\eps^{-\frac32})$,
$C_3(\eps)=O(\eps^{-\frac72})$ when $n=3$.
\end{remark}

The next lemma shows the contractiveness of the mapping $T$.
\begin{lem}\label{lembound6}
There exists an $h_0=o(\eps^\frac{19}{12})$ and
$\rho_0=o(\eps^{\frac{19}{12}}|\log h|^{n-3}h^{\frac{n}2-1})$, such
that for $h\le h_0$, $T$ is a contracting mapping in the ball
$B_h(\rho_0)$ with a contraction factor $\frac12$. That is, for any
$(\mu_h,v_h),\ (\chi_h,w_h)\in B_h(\rho_0)$ there holds
\begin{align}\label{contract}
\|T^{(1)}(\mu_h,v_h)-T^{(1)}(\chi_h,w_h)\|_\lt
&+\frac{1}{\sqrt{\eps}}\|T^{(2)}(\mu_h,v_h)-T^{(2)}(\chi_h,w_h)\|_{H^1} \\
& 
\le \frac{1}{2}\bigl(\|\mu_h-\chi_h\|_\lt
+\frac{1}{\sqrt{\eps}}\|v_h-w_h\|_{H^1}\bigr).  \nonumber
\end{align}
\end{lem}

\begin{proof}
We divide the proof into five steps.

{\em Step 1:} To ease notation, let
\[
T^{(1)}=T^{(1)}(\mu_h,v_h)-T^{(1)}(\chi_h,w_h),\quad
T^{(2)}=T^{(2)}(\mu_h,v_h)-T^{(2)}(\chi_h,w_h).
\]
By the definition of $T^{(i)}$ we get
\begin{align}\label{def1}
&\bigl(T^{(1)},\kappa_h\bigl)+\bigl(\Div(\kappa_h),D(T^{(2)})\bigr)
=0\quad \forall\kappa_h\in W^h_0,\\
&\bigl(\Div(T^{(1)}),Dz_h\bigr)
-\frac{1}{\eps}\bigl(\Phi^\eps D(T^{(2)}),Dz_h\bigr) \label{def2} \\
&\hskip 0.3in =\frac{1}{\eps}\bigl[ \bigl(\Phi^\eps D(w_h-v_h),Dz_h\bigr)
+\bigl(\det(\chi_h)-\det(\mu_h),z_h\bigr)\bigr]\quad \forall z_h\in
V^h_0. \nonumber
\end{align}
Letting $z_h=T^{(2)}$ and $\kappa_h=T^{(1)}$, subtracting
\eqref{def2} from \eqref{def1}, and using the Mean Value
Theorem we have
\begin{align*}
&(T^{(1)},T^{(1)})+\frac{1}{\eps}(\Phi^\eps DT^{(2)},DT^{(2)})\\
&\hspace{0.2in} =\frac{1}{\eps}\bigl[(\Phi^\eps
D(v_h-w_h),DT^{(2)})+(\text{det}(\mu_h)-\text{det}(\chi_h),T^{(2)})\bigr]\\
&\hspace{0.2in} =\frac{1}{\eps}\bigl[(\Phi^\eps
D(v_h-w_h),DT^{(2)})+(\Lambda_h:(\mu_h-\chi_h),T^{(2)})\bigr]\\
&\hspace{0.2in} =\frac{1}{\eps}\bigl[(\Phi^\eps D(v_h-w_h),DT^{(2)})+(\Phi^\eps:
(\mu_h-\chi_h),T^{(2)})\\
&\hspace{0.4in}+((\Lambda_h-\Phi^\eps):(\mu_h-\chi_h),T^{(2)})\bigr]\\
&\hspace{0.2in} =\frac{1}{\eps}\bigl[(\Div (\Phi^\eps
T^{(2)}),D(v_h-w_h))+(\mu_h-\chi_h,\Phi^\eps T^{(2)})\\
&\hspace{0.4in}+((\Lambda_h-\Phi^\eps):(\mu_h-\chi_h),T^{(2)})\bigr]\\
&\hspace{0.2in} =\frac{1}{\eps}\bigl[(\Div (\Pi_h(\Phi^\eps T^{(2)})),D(v_h-w_h))
+(\mu_h-\chi_h,\Phi^\eps T^{(2)})\\
&\hspace{0.4in}+((\Lambda_h-\Phi^\eps):(\mu_h-\chi_h),T^{(2)})\bigr]\\
&\hspace{0.2in} =\frac{1}{\eps}\bigl[(\Phi^\eps T^{(2)}-\Pi_h(\Phi^\eps
T^{(2)}),\mu_h-\chi_h)+((\Lambda_h-\Phi^\eps):(\mu_h-\chi_h),T^{(2)})\bigr]
\end{align*}
\begin{align*}
&\hspace{0.2in}
\le \frac{1}{\eps}\bigl[\|\Phi^\eps T^{(2)}-\Pi_h(\Phi^\eps
T^{(2)})\|_\lt\|\mu_h-\chi_h\|_\lt \\
&\hspace{0.7in}
+C\|\Lambda_h-\Phi^\eps\|_{L^2}\|\mu_h-\chi_h\|_{L^2}\|T^{(2)}\|_{L^\infty}\bigr]
\\
&\hspace{0.2in}
\le \frac{1}{\eps}\bigl[\|\Phi^\eps T^{(2)}-\Pi_h(\Phi^\eps
T^{(2)})\|_\lt\|\mu_h-\chi_h\|_\lt \\
&\hspace{0.7in}
+|\log h|^{3-n}h^{1-\frac{n}2}\|\Lambda_h-\Phi^\eps\|_{L^2}\|\mu_h-\chi_h\|_\lt\|T^{(2)}\|_{H^1}\bigr],
\end{align*}
where $\Lambda_h=\text{cof}(\mu_h+\tau(\chi_h-\mu_h)),\ \tau\in [0,1].$
$n=2,3$. We have used the inverse inequality to get the last
inequality above.

{\em Step 2: The case of $n=2$.} We bound $\|\Phi^\eps-\Lambda_h\|_{L^2}$
as follows:
\begin{align*}
\|\Phi^\eps-\Lambda_h\|_{L^2}&=\|\text{cof}(\sigma^\eps)-\text{cof}(\mu_h+\tau
(\chi_h-\mu_h))\|_{L^2}\\
&=\|\sigma^\eps-\mu_h-\tau( \chi_h-\mu_h)\|_\lt\\
&\le \|\sigma^\eps-I_h\sigma^\eps\|_{L^2}+\|I_h\sigma^\eps-\mu_h\|_{L^2}
+\|\chi_h-\mu_h\|_{L^2}\\
&\le Ch^l\|\sigma^\eps\|_{H^l}+3\rho_0.
\end{align*}

{\em Step 3: The case of $n=3$.} To bound $\|\Phi^\eps-\Lambda_h\|_{L^2}$
in this case, we first write
\begin{align*}
\|(\Phi^\eps-\Lambda_h)_{ij}\|_{L^2}&=\|(\text{cof}(\sigma^\eps)_{ij})
-\text{cof}(\mu_h+\tau(\chi_h-\mu_h))_{ij}\|_{L^2}\\
&=\|\text{det}(\sigma^\eps|_{ij})-\text{det}(\mu_h|_{ij}
+\tau(\chi_h|_{ij}-\mu_h|_{ij}))\|_{L^2},
\end{align*}
where $\sigma|_{ij}$ denotes the $2\times 2$ matrix after deleting
the $i^{th}$ row and $j^{th}$ column. Then, use the Mean Value theorem to get
\begin{align*}
\|(\Phi^\eps-\Lambda_h)_{ij}\|_{L^2}
&=\|\text{det}(\sigma^\eps|_{ij})-\text{det}(\mu_h|_{ij}
+\tau(\chi_h|_{ij}-\mu_h|_{ij}))\|_{L^2}\\
&=\|\Lambda_{ij}:(\sigma^\eps|_{ij}-\mu_h|_{ij}
-\tau(\chi_h|_{ij}-\mu_h|_{ij}))\|_{L^2}\\
&\le \|\Lambda_{ij}\|_{L^\infty}\|\sigma^\eps|_{ij}-\mu_h|_{ij}
-\tau(\chi_h|_{ij}-\mu_h|_{ij})\|_{L^2},
\end{align*}
where
$\Lambda_{ij}=\text{cof}(\sigma^\eps|_{ij}+\lambda(\mu|_{ij}
-\tau(\chi_h|_{ij}-\mu|_{ij})-\sigma^\eps|_{ij})),\ \lambda\in [0,1]$.

On noting that $\Lambda_{ij}\in \mathbf{R}^2$, we have
\begin{align*}
\|\Lambda_{ij}\|_{L^\infty}&=\|\text{cof}(\sigma^\eps|_{ij}
+\lambda(\mu|_{ij}-\tau(\chi_h|_{ij}-\mu|_{ij})-\sigma^\eps|_{ij}))\|_{L^\infty}\\
&=\|\sigma^\eps|_{ij}+\lambda(\mu|_{ij}-\tau(\chi_h|_{ij}-\mu|_{ij})
-\sigma^\eps|_{ij})\|_{L^\infty}\\
&\le C\|\sigma^\eps\|_{L^\infty}\le \frac{C}{\eps}.
\end{align*}
Combining the above estimates gives
\begin{align*}
\|(\Phi^\eps-\Lambda_h)_{ij}\|_{L^2}
&\le \frac{C}{\eps}\|\sigma^\eps|_{ij}-\mu_h|_{ij}
-\tau(\chi_h|_{ij}-\mu_h|_{ij})\|_{L^2}\\
&\le
\frac{C}{\eps}\left(h^l\|\sigma^\eps\|_{H^l}+\rho_0\right).\end{align*}

{\em Step 4:} We now bound $\|\Phi^\eps T^{(2)}-\Pi_h(\Phi^\eps T^{(2)})\|_\lt$ as follows:
\begin{align*}
&\|\Phi^\eps T^{(2)}-\Pi_h(\Phi^\eps T^{(2)})\|_\lt^2
\le Ch^2\|\Phi^\eps T^{(2)}\|_{H^1}^2\\
&\hspace{0.2in} =Ch^2\bigl(\|\Phi^\eps T^{(2)}\|_\lt^2+\|D(\Phi^\eps
T^{(2)})\|_\lt^2\bigr)\\
&\hspace{0.2in} \le Ch^2\bigl(\|\Phi^\eps T^{(2)}\|_\lt^2+\|\Phi^\eps D
T^{(2)}\|_\lt^2+\|D\Phi^\eps T^{(2)}\|_\lt^2\bigr)\\
&\hspace{0.2in} \le Ch^2\bigl(\|\Phi^\eps\|_{L^4}^2
\|T^{(2)}\|_{L^4}^2+\|\Phi^\eps\|_{L^\infty}
\|DT^{(2)}\|_\lt^2+\|D\Phi^\eps\|_{L^3}^2
\|T^{(2)}\|_{L^6}^2\bigr)\\
&\hspace{0.2in} \le Ch^2\bigl(\|\Phi^\eps\|_{L^4}^2
\|T^{(2)}\|_{H^1}^2+\|\Phi^\eps\|_{L^\infty}^2
\|DT^{(2)}\|_\lt^2+\|D\Phi^\eps\|_{L^3}^2
\|T^{(2)}\|_{H^1}^2\bigr)\\
&\hspace{0.2in} \le Ch^2\bigl(\|\Phi^\eps\|_{L^\infty}^2
+\|D\Phi^\eps\|_{L^3}^2\bigr)\|DT^{(2)}\|_\lt^2\\
&\hspace{0.2in} \le \frac{Ch^2}{\eps^{\frac{13}6}}\|DT^{(2)}\|_\lt^2,
\end{align*}
where we have used Sobolev's inequality followed by Poincare's inequality.
Thus,
\begin{align*}
\|\Phi^\eps T^{(2)}-\Pi_h(\Phi^\eps T^{(2)})\|_\lt
\le \frac{Ch}{\eps^{\frac{13}{12}}}\|DT^{(2)}\|_\lt.
\end{align*}

{\em Step 5: Finishing up.} Substituting all estimates from Steps 2-4 into
Step 1, and using the fact that $\Phi^\eps$ is positive definite we obtain
for $n=2,3$
\begin{align*}
\|T^{(1)}\|_\lt^2+\frac{\theta}{\eps}\|DT^{(2)}\|_\lt^2
&\le C\eps^{-{\frac{25}{12}}}\bigl(h+|\log
h|^{3-n}h^{1-\frac{n}2}\rho_0\bigr)\|\mu_h-\chi_h\|_\lt\|DT^{(2)}\|_\lt.
\end{align*}
Using Schwarz's inequality we get
\begin{align*}
\|T^{(1)}\|_\lt+ \frac{1}{\sqrt{\vepsi}} \|T^{(2)}\|_{H^1} &\le
C\eps^{-\frac{19}{12}} \bigl(h+|\log
h|^{3-n}h^{1-\frac{n}2}\rho_0\bigr)\|\mu_h-\chi_h\|_\lt.
\end{align*}

Choosing $h_0=o(\eps^\frac{19}{12})$ and
$\rho_0=o(\eps^\frac{19}{12} |\log h|^{n-3}h^{\frac{n}2-1})$, then
for $h\le h_0$ there holds
\begin{align*}
\|T^{(1)}\|_\lt + \frac{1}{\sqrt{\vepsi}} \|T^{(2)}\|_{H^1}
&\le \frac{1}{2}\|\mu_h-\chi_h\|_\lt\\
&\le \frac{1}{2}\bigl(\|\mu_h-\chi_h\|_\lt
+ \frac{1}{\sqrt{\eps}} \|v_h-w_h\|_{H^1}\bigr).
\end{align*}
The proof is complete.
\end{proof}

We are now ready to state and prove the main theorem of this paper.

\begin{thm}\label{thm1}
Let
$\rho_1=2[C_2(\eps)h^{l-2}+\frac{C_3(\eps)}{\sqrt{\eps}}h^{l-1}](\|\sigma^\eps\|_{H^l}
+\|u^\eps\|_{H^l})$.  Then there exists an $h_1>0$ such that for
$h\le \text{min}\{h_0,h_1\}$, there exists a unique solution
$(\sigma^\eps_h,u^\eps_h)$ to \eqref{prob3}-\eqref{prob4} in the
ball $B_h(\rho_1)$. Moreover,
\begin{align}\label{sigmal2err}
\|\sigma^\eps-\sigma^\eps_h\|_\lt
+\frac{1}{\sqrt{\eps}} \|u^\eps-u^\eps_h\|_{H^1} &\le
 C_4(\eps) h^{l-2}\bigl(\|\sigma^\eps\|_{H^l} +\|u^\eps\|_{H^l}\bigr), \\
\label{sigmah1err}\|\sigma^\eps-\sigma^\eps_h\|_{H^1}&\le
C_5(\eps)h^{l-3}\bigl(\|\sigma^\eps\|_{H^l}
+\|u^\eps\|_{H^l}\bigr).
\end{align}
\end{thm}

\begin{proof}
Let $(\mu_h,v_h)\in B_h(\rho_1)$ and choose $h_1>0$ such that 
\begin{align*}
h_1|\log h_1|^{\frac{2(3-n)}{2l-n}}&\le
C\left(\frac{\eps^{\frac{25}{12}}}{C_3(\eps)(\|\sigma^\eps\|_{H^l}+\|u^\eps\|_{H^l})}\right)^{\frac{2}{2l-n}} \qquad \text{and}\\
h_1|\log h_1|^{\frac{2(3-n)}{2l-n-2}} &\le
C\left(\frac{\eps^{\frac{19}{12}}}{C_2(\eps)(\|\sigma^\eps\|_{H^l}+\|u^\eps\|_{H^l})}\right)^{\frac{2}{2l-n-2}}.
\end{align*}
Then $h\le \text{min}\{h_0,h_1\}$ implies $\rho_1\le \rho_0$. Thus,
using the triangle inequality and Lemmas \ref{lembound5} and
\ref{lembound6} we get
\begin{align*}
&\|I_h\sigma^\eps-T^{(1)}(\mu_h,v_h)\|_\lt
+\frac{1}{\sqrt{\eps}}\|I_hu^\eps-T^{(2)}(\mu_h,v_h)\|_{H^1}
\le \|I_h\sigma^\eps-T^{(1)}(I_h\sigma^\eps,I_hu^\eps)\|_\lt\\
&\hspace{0.4in}
+\|T^{(1)}(I_h\sigma^\eps,I_hu^\eps)-T^{(1)}(\mu_h,v_h)\|_\lt
+\frac{1}{\sqrt{\eps}}\|I_h u^\eps-T^{(2)}(I_h\sigma^\eps,I_hu^\eps)\|_{H^1}\\
&\hspace{0.4in}
+\frac{1}{\sqrt{\eps}}\|T^{(2)}(I_h\sigma^\eps,I_hu^\eps)-T^{(2)}(\mu_h,v_h)\|_{H^1}\\
&\quad \le
\bigl[C_2(\eps)h^{l-2}+\frac{C_3(\eps)}{\sqrt{\eps}}h^{l-1} \bigr]
\bigl(\|\sigma^\eps\|_{H^l}
+\|u^\eps\|_{H^l}\bigr) \\
&\hspace{0.4in} +\frac{1}{2}\bigl(\|I_h\sigma^\eps-\mu_h\|_\lt
+\frac{1}{\sqrt{\vepsi}}\|I_hu^\eps-v_h\|_{H^1}\bigr)\\
&\quad \le \frac{\rho_1}{2}+\frac{\rho_1}{2}=\rho_1<1.
\end{align*}
So $T(\mu_h,v_h)\in B_h(\rho_1)$. Clearly, $T$
is a continuous mapping. Thus, $T$ has a unique fixed point
$(\sigma^\eps_h,u^\eps_h)\in B_h(\rho_1)$ which is the unique
solution to \eqref{prob3}-\eqref{prob4}.

Next, we use the triangle inequality to get
\begin{align*}
\|\sigma^\eps-\sigma^\eps_h\|_\lt+\frac{1}{\sqrt{\eps}}\|u^\eps-u^\eps_h\|_{H^1}&\le \|\sigma^\eps-I_h\sigma^\eps\|_\lt+\|I_h\sigma^\eps-\sigma^\eps_h\|_\lt\\
&\hspace{0.4in}+\frac{1}{\sqrt{\eps}}\bigl(\|u^\eps-I_h
u^\eps\|_{H^1}+\|I_h
u^\eps-u^\eps_h\|_{H^1}\bigr)\\
&\le
\rho_1+Ch^{l-1}\bigl(\|\sigma^\eps\|_{H^l}+\|u^\eps\|_{H^l}\bigr)\\
&\le
C_4(\eps)h^{l-2}\bigl(\|\sigma^\eps\|_{H^l}+\|u^\eps\|_{H^l}\bigr).
\end{align*}

Finally, using the inverse inequality we have
\begin{align*}
\|\sigma^\eps-\sigma^\eps_h\|_{H^1}
&\le \|\sigma^\eps-I_h\sigma^\eps\|_{H^1}
+\|I_h\sigma^\eps-\sigma^\eps_h\|_{H^1}\\
&\le \|\sigma^\eps-I_h\sigma^\eps\|_{H^1}
+Ch^{-1}\|I_h\sigma^\eps-\sigma^\eps_h\|_\lt\\
&\le Ch^{l-1}\|\sigma^\eps\|_{H^l}+Ch^{-1}\rho_1\\
&\le C_5(\eps)h^{l-3}\bigl[\|\sigma^\eps\|_{H^l}
+\|u^\eps\|_{H^l}\bigr].
\end{align*}
The proof is complete.
\end{proof}

\begin{remark}
By the definition of $\rho_1$, and the remark following Lemma
\ref{lembound5}, we see that
$C_4(\eps)=C_5(\eps)=O(\eps^{-\frac52})$ when $n=2$,
$C_4(\eps)=C_5(\eps)=O(\eps^{-4})$ when $n=3$.
\end{remark}

Comparing with error estimates for the linearized problem in
Theorem \ref{thm0}, we see that the above $H^1$-error for the
scalar variable is not optimal. Next, we shall employ a similar
duality argument as used in the proof of Theorem \ref{thm0}
to show that the estimate can be improved to optimal order.

\begin{thm}\label{uh1errthm}
Under the same hypothesis of Theorem \ref{thm1} there holds
\begin{equation} \label{uh1err}
\|u^\eps-u^\eps_h\|_{H^1} \le C_4(\eps) \eps^{-2} \bigl[
h^{l-1}+C_5(\eps)h^{2(l-2)}\bigr]
\bigl(\|\sigma^\eps\|_{H^l}+\|u^\eps\|_{H^l}\bigr).
\end{equation}
\end{thm}

\begin{proof}
The regularity assumption implies that there exists $(\kappa,z)
\in W_0\times V_0\cap H^3(\Ome)$ such that
\begin{alignat}{2}
\label{adjointh1thm1}(\kappa,\mu)+(\Div(\mu),Dz)&=0\spa &&\forall
\mu\in W_0,\\
\label{adjointh1thm2}(\Div(\kappa),Dv)-\frac{1}{\eps}(\Phi^\eps
Dz,Dv)&=\frac{1}{\eps}(D(u^\eps-u^\eps_h),Dv)\spa &&\forall v\in
V_0,\end{alignat} with
\begin{equation}\label{zboundh1thm}\|z\|_{H^3}\le
C_b(\eps)\|D(u^\eps-u^\eps_h)\|_\lt.
\end{equation}
It is easy to check that $\sigma^\eps-\sigma^\eps_h$ and
$u^\eps-u^\eps_h$ satisfy the following error equations:
\begin{alignat}{2}
\label{erroreqn1}(\sigma^\eps-\sigma^\eps_h,\mu_h)
+(\Div(\mu_h),D(u^\eps-u^\eps_h))&=0\spa &&\forall \mu_h\in W^h_0,\\
\label{erroreqn2}(\Div(\sigma^\eps-\sigma^\eps_h),Dv_h)
+\frac{1}{\eps}(\text{det}(\sigma^\eps)-\text{det}(\sigma^\eps_h),v_h)&=0\spa
&&\forall v_h\in V^h_0.
\end{alignat}
By \eqref{adjointh1thm1}-\eqref{erroreqn2} and the Mean Value Theorem we get
\begin{align*}
&\frac{1}{\eps}\|D(u^\eps-u^\eps_h)\|_\lt^2
=\bigl(\Div(\kappa),D(u^\eps-u^\eps_h)\bigr)
-\frac{1}{\eps}\bigl(\Phi^\eps Dz,D(u^\eps-u^\eps_h)\bigr)\\
&\quad =\bigl(\Div(\Pi_h\kappa),D(u^\eps-u^\eps_h)\bigr)
-\frac{1}{\eps}\bigl(\Phi^\eps D(u^\eps-u^\eps_h),Dz\bigr)
+\bigl(\Div(\kappa-\Pi_h\kappa),D(u^\eps-u^\eps_h)\bigr)\\
&\quad =\bigl(\sigma^\eps_h-\sigma^\eps,\Pi_h\kappa\bigr)
-\frac{1}{\eps}\bigl(\Phi^\eps D(u^\eps-u^\eps_h),Dz\bigr)
+\bigl(\Div(\kappa-\Pi_h\kappa),D(u^\eps-u^\eps_h)\bigr)\\
&\quad =\bigl(\sigma^\eps_h-\sigma^\eps,\kappa\bigr)
-\frac{1}{\eps}\bigl(\Phi^\eps D(u^\eps-u^\eps_h),Dz\bigr)\\
&\hspace{0.7in}+\bigl(\Div(\kappa-\Pi_h\kappa),D(u^\eps-I_hu^\eps)\bigr)
+\bigl(\sigma^\eps_h-\sigma^\eps,\Pi_h\kappa-\kappa\bigr)\\
&\quad =\bigl(\Div(\sigma^\eps-\sigma^\eps_h),Dz\bigr)
-\frac{1}{\eps}\bigl(\Phi^\eps D(u^\eps-u^\eps_h),Dz\bigr)\\
&\hspace{0.7in}+\bigl(\Div(\kappa-\Pi_h\kappa),D(u^\eps-I_hu^\eps)\bigr)
+\bigl(\sigma^\eps_h-\sigma^\eps,\Pi_h\kappa-\kappa\bigr)\\
&\quad =\bigl(\Div(\sigma^\eps-\sigma^\eps_h),D(z-I_hz)\bigr)
-\frac{1}{\eps}\bigl(\Phi^\eps D(u^\eps-u^\eps_h),D(z-I_hz)\bigr)\\
&\hspace{0.7in}+\bigl(\Div(\kappa-\Pi_h\kappa),D(u^\eps-I_hu^\eps)\bigr)
+\bigl(\sigma^\eps_h-\sigma^\eps,\Pi_h\kappa-\kappa\bigr)
\\
&\hspace{0.7in}-\frac{1}{\eps}\bigl(\text{det}(\sigma^\eps)
-\text{det}(\sigma^\eps_h),I_hz\bigr)-\frac{1}{\eps}\bigl(\Phi^\eps D(u^\eps-u^\eps_h),D(I_hz)\bigr)\\
&\quad =\bigl(\Div(\sigma^\eps-\sigma^\eps_h),D(z-I_hz)\bigr)
-\frac{1}{\eps}\bigl(\Phi^\eps D(u^\eps-u^\eps_h),D(z-I_hz)\bigr)\\
&\hspace{0.7in}+\bigl(\Div(\kappa-\Pi_h\kappa),D(u^\eps-I_hu^\eps)\bigr)
+\bigl(\sigma^\eps_h-\sigma^\eps,\Pi_h\kappa-\kappa\bigr)\\
&\hspace{0.7in}-\frac{1}{\eps}\bigl(\Psi^\eps:(\sigma^\eps-\sigma^\eps_h),I_hz\bigr)
-\frac{1}{\eps}\bigl(\Phi^\eps D(u^\eps-u^\eps_h),D(I_hz)\bigr),
\end{align*}
where $\Psi^\eps=\text{cof}(\sigma^\eps+\tau[\sigma^\eps_h-\sigma^\eps])$
for $\tau\in [0,1]$.

Next, we note that
\begin{align*}
&\bigl(\Psi^\eps:(\sigma^\eps-\sigma^\eps_h),I_hz\bigr)+\bigl(\Phi^\eps
D(u^\eps-u^\eps_h),D(I_hz)\bigr)\\
&\quad =\bigl(\Phi^\eps:(\sigma^\eps-\sigma^\eps_h),I_hz\bigr)
+\bigl(\Div(\Phi^\eps I_h z),D(u^\eps-u^\eps_h)\bigr)
+\bigl((\Psi^\eps-\Phi^\eps):(\sigma^\eps-\sigma^\eps_h),I_hz\bigr)\\
&\quad =\bigl(\sigma^\eps-\sigma^\eps_h),\Phi^\eps I_hz\bigr)
+\bigl(\Div(\Pi_h(\Phi^\eps I_h z)),D(u^\eps-u^\eps_h)\bigr)
+\bigl((\Psi^\eps-\Phi^\eps):(\sigma^\eps-\sigma^\eps_h),I_hz\bigr)\\
&\hspace{0.7in}
+\bigl(\Div(\Phi^\eps I_hz-\Pi_h(\Phi^\eps I_hz)),D(u^\eps-I_hu^\eps)\bigr)\\
&\quad =\bigl(\sigma^\eps-\sigma^\eps_h,\Phi^\eps I_hz-\Pi_h(\Phi^\eps I_hz)\bigr)
+\bigl((\Psi^\eps-\Phi^\eps):(\sigma^\eps-\sigma^\eps_h),I_hz\bigr)\\
&\hspace{0.7in} +\bigl(\Div(\Phi^\eps I_hz-\Pi_h(\Phi^\eps
I_hz)),D(u^\eps-I_hu^\eps)\bigr).
\end{align*}
Using this and the same technique used in Step 4 of Lemma \ref{lembound6} we have
\begin{align*}
&\frac{1}{\eps}\|D(u^\eps-u^\eps_h)\|_\lt^2
=\bigl(\Div(\sigma^\eps-\sigma^\eps_h),D(z-I_hz)\bigr)
-\frac{1}{\eps}\bigl(\Phi^\eps D(u^\eps-u^\eps_h),D(z-I_hz)\bigr)\\
&\hspace{0.2in}
+\frac{1}{\eps}\Bigl[\bigl((\Phi^\eps-\Psi^\eps):(\sigma^\eps-\sigma^\eps_h),I_hz\bigr)
+\bigl(\sigma^\eps-\sigma^\eps_h,\Pi_h(\Phi^\eps I_hz)-\Phi^\eps I_hz\bigr)\\
&\hspace{0.2in} +\bigl(\Div(\Pi_h(\Phi^\eps I_hz)-\Phi^\eps
I_hz),D(u^\eps-I_h u^\eps)\bigr)\Bigr]
+\bigl(\sigma^\eps_h-\sigma^\eps,\Pi_h\kappa-\kappa\bigr)\\
&\hspace{0.2in}
+\bigl(\Div(\kappa-\Pi_h\kappa),D(u^\eps-I_hu^\eps)\bigr)
\\
&\le \bigl[\|\Div(\sigma^\eps-\sigma^\eps_h)\|_\lt
+\frac{C}{\eps^2}\|D(u^\eps-u^\eps_h)\|_\lt\bigr]\|D(z-I_hz)\|_\lt\\
&\hspace{0.2in}
+\frac{C}{\eps}\bigl[\|\Phi^\eps-\Psi^\eps\|_\lt\|\sigma^\eps-\sigma^\eps_h\|_\lt\|I_hz\|_{L^\infty}
+\|\sigma^\eps-\sigma^\eps_h\|_\lt\|\Pi_h(\Phi^\eps I_hz)-\Phi^\eps I_hz\|_\lt\\
&\hspace{0.2in} +\|\Div(\Pi_h(\Phi^\eps I_hz)-\Phi^\eps I_hz)\|_\lt
\|D(u^\eps-I_hu^\eps)\|_\lt\bigr]
+\|\kappa-\Pi_h\kappa\|_\lt \|\sigma^\eps-\sigma^\eps_h\|_\lt\\
&\hspace{0.2in}
+\|\Div(\kappa-\Pi_h\kappa)\|_\lt\|D(u^\eps-I_hu^\eps)\|_\lt
\\
&\le C h^2\bigl(\|\sigma^\eps-\sigma^\eps_h\|_{H^1}
+\frac{1}{\eps^2}\|u^\eps-u^\eps_h\|_{H^1}\bigr)\|z\|_{H^3}\\
&\hspace{0.2in}
+\frac{C}{\eps^2}\bigl(\|\Phi^\eps-\Psi^\eps\|_\lt\|\sigma^\eps
-\sigma^\eps_h\|_\lt+h\|\sigma^\eps-\sigma^\eps_h\|_\lt
+\|u^\eps-I_hu^\eps\|_{H^1}\bigr)\|z\|_{H^3}\\
&\hspace{0.2in}
+Ch\|\sigma^\eps-\sigma^\eps_h\|_\lt\|\kappa\|_{H^1}
+C\|u^\eps-I_hu^\eps\|_{H^1}\|\kappa\|_{H^1} \\
&\le \Bigl\{\frac{(C_4(\eps)+C_5(\eps))h^{l-1}}{\eps^{\frac32}}\bigl[\|\sigma^\eps\|_{H^l}
+\|u^\eps\|_{H^l}\bigr]+\frac{C_4(\eps)h^{l-2}}{\eps^2}\|\Phi^\eps-\Psi^\eps\|_\lt\Bigr\} \|z\|_{H^3}\\
&\le C_b(\eps)\Bigl\{\frac{(C_4(\eps)+C_5(\eps))h^{l-1}}{\eps^{\frac32}}\bigl[\|\sigma^\eps\|_{H^l}+\|u^\eps\|_{H^l}\bigr] \\
&\hspace{0.2in}
+\frac{C_4(\eps)h^{l-2}}{\eps^2}\|\Phi^\eps-\Psi^\eps\|_\lt\Bigr\}\|D(u^\eps-u^\eps_h)\|_\lt.
\end{align*}

We now bound $\|\Phi^\eps-\Psi^\eps\|_\lt$ separately for the cases $n=2$ and
$n=3$.  First, when $n=2$ we have
\begin{align*}
\|\Phi^\eps-\Psi^\eps\|_\lt&=\|\text{cof}(\sigma^\eps)-\text{cof}(\sigma^\eps_h+\tau[\sigma^\eps-\sigma^\eps_h])\|_\lt\\
&=\|\sigma^\eps-(\sigma^\eps_h+\tau[\sigma^\eps-\sigma^\eps_h])\|_\lt\\
&\le C_4(\eps)h^{l-2}\bigl[\|\sigma^\eps\|_{H^l}+\|u^\eps\|_{H^l}\bigr].
\end{align*}
Second, when $n=3$, on noting that
\begin{align*}
|(\Phi^\eps-\Psi^\eps)_{ij}|
&=|(\text{cof}(\sigma^\eps))_{ij}-(\text{cof}(\sigma^\eps_h+\tau[\sigma^\eps-\sigma^\eps_h]))_{ij}|\\
&=|\text{det}(\sigma^\eps|_{ij})-\text{det}(\sigma^\eps|_{ij}
+\tau[\sigma^\eps|_{ij}-\sigma^\eps_h|_{ij}])|,
\end{align*}
and using the Mean Value Theorem and Sobolev inequality we get
\begin{align*}
\|(\Psi^\eps)_{ij}-(\Phi^\eps)_{ij}\|_\lt
&=(1-\tau)\|(\Lambda^\eps)^{ij}:(\sigma^\eps|_{ij}-\sigma^\eps_h|_{ij})\|_\lt\\
&\le \|(\Lambda^\eps)^{ij}\|_{H^1}\|\sigma^\eps|_{ij}-\sigma^\eps_h|_{ij}\|_{H^1},
\end{align*}
where
$(\Lambda^\eps)^{ij}=\text{cof}(\sigma^\eps|_{ij}+\lambda[\sigma^\eps_h|_{ij}
-\sigma^\eps|_{ij}])$ for $\lambda\in [0,1]$.  Since $(\Lambda^\eps)^{ij}\in
\mathbf{R}^{2\times 2}$, then
\begin{align*}
\|(\Lambda)^{ij}\|_{H^1}=\|\sigma^\eps|_{ij}+\lambda(\sigma^\eps_h|_{ij}
-\sigma^\eps|_{ij})\|_{H^1}\le C\|\sigma^\eps\|_{H^1}=O(\eps^{-1}).
\end{align*}
Thus,
\begin{align*}
\|\Phi^\eps-\Psi^\eps\|_\lt \le C_4(\eps) \eps^{-1} h^{l-2}
\bigl(\|\sigma^\eps\|_{H^l}+\|u^\eps\|_{H^l}\bigr).
\end{align*}

Finally, combining the above estimates
we obtain
\begin{align*}
\|D(u^\eps-u^\eps_h)\|_\lt \le C_4(\eps) \eps^{-2} \bigl[ h^{l-1}
+C_4(\eps) h^{2(l-2)}\bigr] \bigl(\|\sigma^\eps\|_{H^l}
+\|u^\eps\|_{H^l}\bigr).
\end{align*}
We note that $2(l-2)\geq l-1$ for $k\geq 2$. The proof is complete.
\end{proof}

\section{Numerical experiments and rates of convergence}\label{sec-5}

In this section, we provide several $2$-D numerical experiments
to gauge the efficiency of the mixed finite element method developed in
the previous sections. We numerically determine the ``best'' choice of
the mesh size $h$ in terms of $\eps$, and rates of convergence for
both $u^0-u^\eps$ and $u^\eps-u^\eps_h$. All tests given below are done
on domain $\Ome=[0,1]^2$. We refer the reader to \cite{Feng2,Neilan_thesis} 
for more extensive $2$-D and $3$-D numerical simulations. We like to 
remark that the mixed finite element methods we tested are often 
$10$--$20$ times faster than the Aygris finite element Galerkin 
method studied in \cite{Feng3}.

\subsection*{Test 1:}

For this test, we calculate $\|u^0-u^\eps_h\|$ for fixed $h=0.015$,
while varying $\epsilon$ in order to estimate $\|u^\eps-u^0\|$. We
use quadratic Lagrange element for both variables and solve problem
\eqref{prob1}--\eqref{prob2} with the following test functions:
%
%
\begin{alignat*}{4}
&\text{(a). } u^0 =\frac14 e^{\frac{x^2+y^2}2}, &&\quad
f=(1+x^2+y^2)e^{\frac{x^2+y^2}2}, &&\quad g=e^{\frac{x^2+y^2}2},\\
&\text{(b). } u^0 =x^4+y^2, &&\quad f=24x^2, &&\quad g=x^4+y^2.
\end{alignat*}

After having computed the error, we divide it by various powers of
$\epsilon$ to estimate the rate at which each norm converges.
Tables \ref{test1ae1} and \ref{test1be1} clearly show that
$\|\sigma^0-\sigma^\eps_h\|_\lt=O(\eps^{\frac14})$. Since
$h$ is very small, we then have
$\|u^0-u^\eps\|_\htw\approx \|\sigma^0-\sigma^\eps_h\|_\lt
=O(\eps^{\frac14})$.  Based on this heuristic argument, we predict
that $\|u^0-u^\eps\|_\htw=O(\eps^{\frac14})$.  Similarly, from
Tables \ref{test1ae1} and \ref{test1be1}, we see that
$\|u^0-u^\eps\|_\lt\approx O(\eps)$ and
$\|u^0-u^\eps\|_{H^1}\approx O(\eps^{\frac12})$.
\begin{figure}[tbh]
\begin{center}
\includegraphics[angle=0,width=5cm,height=4cm]{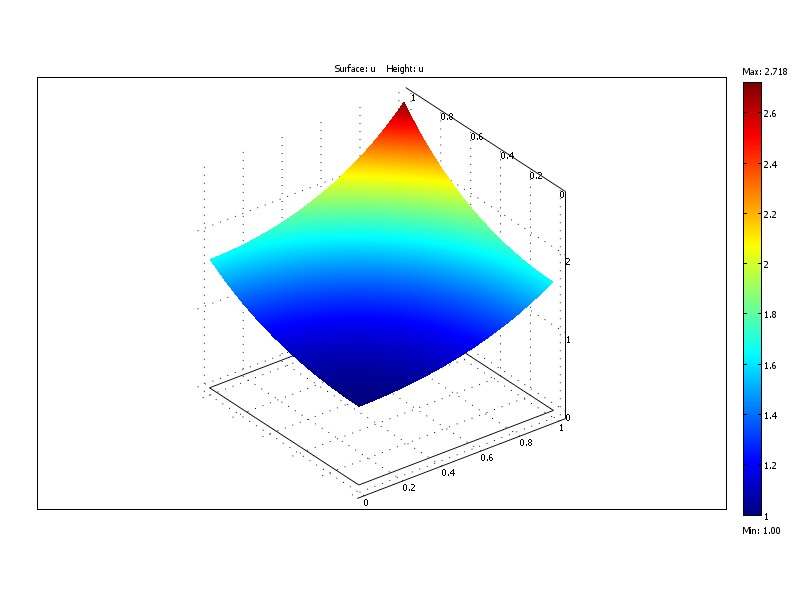}
\includegraphics[angle=0,width=5cm,height=4cm]{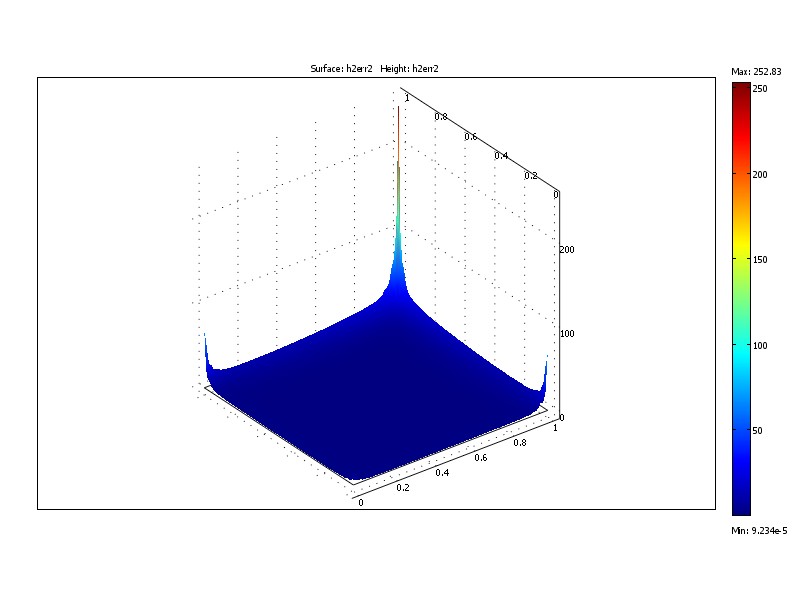}
\end{center}
\caption{{\scriptsize Test 1a.  Computed solution $u^\vepsi_h$ (left)
and its $L^2$-error (right)\, ($\eps$=0.05)}}\label{figpic1}
\end{figure}

\begin{table}[tbh]
\begin{center}
\begin{tabular}{|l|l|l|l|}\hline
$\eps$ & $\|u_h^\eps-u^0\|_\lt$ & $\|u_h^\eps-u^0\|_{H^1}$ &
$\|\sigma^\eps_h-\sigma^0\|_\lt$\\
\hline0.75&    0.031968735& 0.168237927 &1.412579201\\
\hline0.5 &0.038716921 &0.196397556 &1.559234748\\
\hline0.25& 0.040987803 &0.206004854 &1.644877503\\
\hline0.1 &0.032218007 &0.168139823 &1.541246898\\
\hline0.075 &0.028113177& 0.150389494& 1.480968264\\
\hline0.05 &0.022258985& 0.124863926& 1.386775396\\
\hline0.025 &0.013676045& 0.086203248& 1.217747100\\
\hline0.0125 &0.007816727 &0.057280014 &1.052222885\\
\hline0.005 &0.003511072 &0.032109189& 0.853140082\\
\hline0.0025 &0.001863935& 0.020252025& 0.722844382\\
\hline0.00125 &0.000973479 &0.012568349 &0.611218455\\
\hline0.0005 &0.000404799& 0.006544116& 0.492454059\\
\hline
\end{tabular}
\caption{{\scriptsize Test 1a: Change of $\|u^0-u^\eps_h\|$
w.r.t. $\eps$\, ($h=0.015$)}} \label{test1a}
\end{center}
\end{table}

\begin{table}[tbh]
\begin{center}
\begin{tabular}{|l|l|l|l|}\hline
$\eps$ & $\frac{\|u_h^\eps-u^0\|_\lt}{\eps}$ & $\frac{\|u_h^\eps-u^0\|_{H^1}}{\sqrt{\eps}}$ & $\frac{\|\sigma^\eps_h-\sigma^0\|_\lt}{\sqrt[4]{\eps}}$\\
\hline0.75  &  0.04262498 & 0.194264425& 1.517915136\\
\hline0.5& 0.077433843 &0.277748087 &1.854253057\\
\hline0.25  &  0.163951212& 0.412009709& 2.326208073\\
\hline0.1 &0.322180074 &0.531704805& 2.740767624\\
\hline0.075&   0.374842355& 0.54914479 & 2.829960907\\
\hline0.05  &  0.445179694& 0.558408453& 2.932672906\\
\hline0.025  & 0.54704179& 0.545197212& 3.062471825\\
\hline0.0125  &0.625338155 &0.51232802  &3.146880418\\
\hline 0.005&   0.702214497&0.454092502 &3.208321232\\
\hline 0.0025&  0.745574141 &0.405040492 &3.232658349\\
\hline0.00125 &0.778783297 &0.355486596 &3.250640603\\
\hline0.0005 & 0.809598913& 0.29266175 & 3.293238774\\
\hline\end{tabular}
\caption{{\scriptsize Test 1a: Change of $\|u^0-u^\eps_h\|$ w.r.t.
$\eps$\, ($h=0.015$)}} \label{test1ae1}
\end{center}
\end{table}

\begin{table}[tbh]
\begin{center}
\begin{tabular}{|l|l|l|l|}\hline
$\eps$ & $\|u_h^\eps-u^0\|_\lt$ & $\|u_h^\eps-u^0\|_{H^1}$ &
$\|\sigma^\eps_h-\sigma^0\|_\lt$\\
\hline0.75  &  0.080523289& 0.441995475& 3.65931592\\
\hline0.5& 0.082589346& 0.448160685 &3.706413496\\
\hline0.25  &  0.074746237& 0.412192916& 3.603993202\\
\hline0.1& 0.051429563 &0.309140745 &3.233364656\\
\hline0.075  & 0.043554563 &0.273452007& 3.091264143\\
\hline0.05  &  0.033436507& 0.226024335& 2.885806127\\
\hline0.025 &  0.020115546& 0.158107558& 2.538905473\\
\hline0.0125  &0.011590349& 0.107777549 &2.211633785\\
\hline0.005 &  0.005376049& 0.06303967 & 1.820550192\\
\hline0.0025 & 0.002939459& 0.041182521 &1.559730105\\
\hline0.00125 &0.001580308 &0.026467488 &1.330131572\\
\hline0.0005&  0.000679181& 0.014385878& 1.075465946\\
\hline
\end{tabular}
\caption{{\scriptsize Test 1b: Change of $\|u^0-u^\eps_h\|$
w.r.t. $\eps$\, ($h=0.015$)}} \label{test1b}
\end{center}
\end{table}

\begin{table}[tbh]
\begin{center}
\begin{tabular}{|l|l|l|l|}\hline
$\eps$ & $\frac{\|u_h^\eps-u^0\|_\lt}{\eps}$ & $\frac{\|u_h^\eps-u^0\|_{H^1}}{\sqrt{\eps}}$ & $\frac{\|\sigma^\eps_h-\sigma^0\|_\lt}{\sqrt[4]{\eps}}$\\
\hline0.75  &  0.107364385& 0.510372413& 3.932190858\\
\hline0.5& 0.165178691 &0.63379492&  4.4076933\\
\hline0.25   & 0.298984949& 0.824385832 &5.096816065\\
\hline0.1& 0.514295635& 0.977588871& 5.749825793\\
\hline0.075 &  0.580727513& 0.99850555&  5.907052088\\
\hline0.05  &  0.668730140& 1.010811555& 6.102736940\\
\hline0.025 &  0.804621849& 0.999959999 &6.385009233\\
\hline0.0125 & 0.927227955& 0.963991701& 6.614327771\\
\hline0.005 &  1.075209747& 0.891515564& 6.846366682\\
\hline0.0025  &1.175783722& 0.823650411& 6.975325082\\
\hline0.00125 &1.264246558 &0.748613599& 7.074033284\\
\hline0.0005 & 1.358362838 &0.643356045& 7.192074244\\
\hline
\end{tabular}
\caption{{\scriptsize Test 1b: Change of $\|u^0-u^\eps_h\|$ w.r.t. $\eps$\,
($h=0.015$)}} \label{test1be1}
\end{center}
\end{table}

\subsection*{Test 2:}
The purpose of this test is to calculate the rate of convergence of
$\|u^\eps-u^\eps_h\|$ for fixed $\eps$ in various norms. We use
quadratic Lagrange element for both variables and solve
problem \eqref{prob1}--\eqref{prob2} with boundary
condition $D^2u^\vepsi\nu\cdot\nu =\vepsi$ on $\p\Ome$ being
replaced by $D^2u^\vepsi\nu\cdot\nu =h_\vepsi$ on $\p\Ome$
and using the following test functions:
%
%
\begin{alignat*}{2}
\text{(a). }  u^\eps &=20x^6+y^6,
&&\quad f^\eps=18000x^4y^4-\eps(7200x^2+360y^2), \\
g^\eps &=20x^6+y^6, &&\quad h^\eps=600x^4\nu_x^2+30y^4\nu_y^2.\\
\text{(b). }  u^\eps &=x\text{sin}(x)+y\text{sin}(y),
&&\quad f^\eps=(2\text{cos}(x)-x\text{sin}(x))(2\text{cos}(y)-y*\text{sin}(y))
\\
& &&\qquad\qquad
 -\eps(x\text{sin}(x)-4\text{cos}(x)+y\text{sin}(y)-4\text{cos}(y)),\\
g^\eps &=x\text{sin}(x)+y\text{sin}(y),
&&\quad h^\eps=(2\text{cos}(x)-x\text{sin}(x))\nu_x^2+(2\text{cos}(y)-y\text{sin}(y))\nu_y^2.
\end{alignat*}

After having computed the error in different norms, we divided
each value by a power of $h$ expected to be the convergence rate
by the analysis in the previous section.  As seen from Tables
\ref{test2a2} and \ref{test2b2}, the error converges exactly as
expected in $H^1$-norm, but $\sigma^\eps_h$ appears to converge
one order of $h$ better than the analysis shows. In addition, the
error seems to converge optimally in $L^2$-norm although a
theoretical proof of such a result has not yet been proved.
\begin{figure}[tbh]
\begin{center}
\includegraphics[angle=0,width=5cm,height=3.85cm]{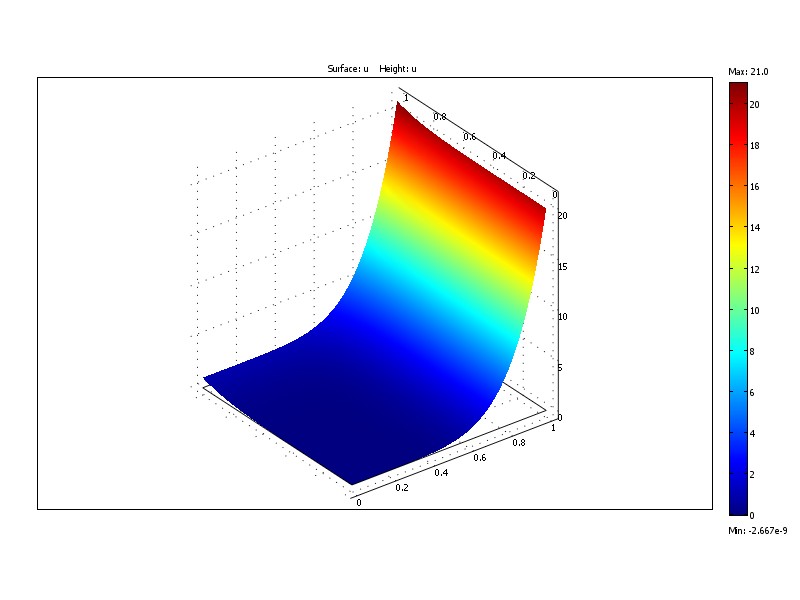}
\includegraphics[angle=0,width=5cm,height=3.85cm]{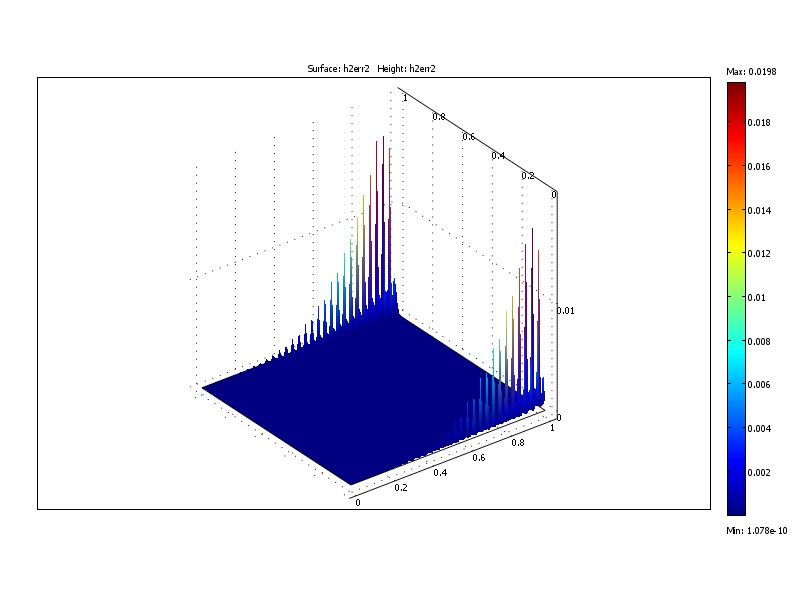}
\end{center}
\caption{{\scriptsize Test 1a.  Computed Solution $u^\vepsi_h$ and its
$L^2$-error\, ($h=0.015$)}} \label{figpic2}
\end{figure}

\begin{table}[tbh]
\begin{center}
\begin{tabular}{|l|c|c|c|c|c|}\hline $n$ & $h$ &
$\|u^\eps-u^\eps_h\|_\lt$ & $\|u^\eps-u^\eps_h\|_{H^1}$ &
$\|\sigma^\eps-\sigma^\eps_h\|_\lt$ &
$\|\sigma^\eps-\sigma^\eps_h\|_{H^1}$\\
\hline10  &0.1 &0.004334849 &0.335913679& 0.083695878 &5.995796194\\
\hline20& 0.05& 0.000545090&  0.084457090  &0.011891926 &1.813405912\\
\hline30& 0.033333333& 0.000161694 &0.037576588& 0.003840822&
0.916912755\\
\hline40& 0.025& 6.82423E-05& 0.021145181 &0.001747951 &0.574128035\\
\hline50& 0.02 &3.49467E-05& 0.013535235& 0.000959941& 0.403471189\\
\hline
\end{tabular}
\caption{{\scriptsize Test 2a: Change of $\|u^\eps-u^\eps_h\|$ w.r.t. $h$ \,
($\eps=0.001$)}} \label{test2a}
\end{center}
\end{table}

\begin{table}[tbh]
\begin{center}
\begin{tabular}{|l|c|c|c|c|}\hline $n$ & $h$& $\frac{\|u^\eps-u^\eps_h\|_\lt}{h^3}$ &
$\frac{\|u^\eps-u^\eps_h\|_{H^1}}{h^2}$ &
$\frac{\|\sigma^\eps-\sigma^\eps_h\|_\lt}{h}$\\
\hline10 & 0.1 &4.334849478& 33.59136794& 0.83695878\\
\hline20&  0.05 &4.360719207& 33.78283588& 0.237838517\\
\hline30 & 0.033333333& 4.365734691& 33.81892961& 0.115224646\\
\hline40 & 0.025& 4.367510244& 33.83229037& 0.069918055\\
\hline50  &0.02 &4.368335996& 33.83808875& 0.047997042\\
\hline
\end{tabular}
\caption{{\scriptsize Test 2a: Change of $\|u^\eps-u^\eps_h\|$ w.r.t.
$h$\, ($\eps=0.001$)}} \label{test2a2}
\end{center}\end{table}

\begin{table}[tbh]
\begin{center}
\begin{tabular}{|l|c|c|c|c|c|}\hline $n$ & $h$ &
$\|u^\eps-u^\eps_h\|_\lt$ & $\|u^\eps-u^\eps_h\|_{H^1}$ &
$\|\sigma^\eps-\sigma^\eps_h\|_\lt$ & $\|\sigma^\eps-\sigma^\eps_h\|_{H^1}$\\
\hline10  &0.1& 1.34918E-05 &0.001045141& 6.86623E-05& 0.005995181\\
\hline20 &0.05 &1.68723E-06& 0.000261390 & 1.19992E-05 &0.002165423\\
\hline30 &0.033333333 &4.99964E-07& 0.000116182& 4.33789E-06&
0.001185931\\
\hline40 &0.025 &2.10928E-07 &6.53541E-05 &2.10913E-06& 0.000772419\\
\hline50& 0.02 &1.07997E-07 &4.18271E-05 &1.20594E-06& 0.000553558\\
\hline
\end{tabular}
\caption{{\scriptsize Test 2b: Change of $\|u^\eps-u^\eps_h\|$ w.r.t.
$h$\, ($\eps=0.001$)}} \label{test2b}
\end{center}
\end{table}

\begin{table}[tbh]
\begin{center}
\begin{tabular}{|l|c|c|c|c|}\hline $n$ & $h$ &
$\frac{\|u^\eps-u^\eps_h\|_\lt} {h^3}$ &
$\frac{\|u^\eps-u^\eps_h\|_{H^1}}{h^2}$ &
$\frac{\|\sigma^\eps-\sigma^\eps_h\|_\lt}{h}$\\ \hline10 & 0.1&
0.013491783& 0.104514066& 0.000686623\\
\hline20 & 0.05  &  0.013497875& 0.104556106& 0.000239985\\
\hline30& 0.033333333& 0.013499020&  0.104563889 &0.000130137\\
\hline40 &0.025& 0.013499422& 0.104566604 &8.43651E-05\\
\hline50  &0.02 &0.013499606& 0.104567861& 6.02971E-05\\
\hline
\end{tabular}
\caption{{\scriptsize Test 2b: Change of $\|u^\eps-u^\eps_h\|$ w.r.t.
$h$\, ($\eps=0.001$)}} \label{test2b2}
\end{center}
\end{table}

\subsection*{Test 3}
In this test, we fix a relation between $\epsilon$ and $h$, and then
determine the ``best'' choice for $h$ in terms of $\epsilon$ such that
the global error $u^0-u^\eps_h$ has the same convergence rate as
that of $u^0-u^\eps$. We solve problem \eqref{prob1}--\eqref{prob2}
with the following test functions:
\begin{alignat*}{4}
\text{(a). } u^0 &=x^4+y^2, &&\quad f=24x^2, &&\quad g=x^4+y^2.\\
\text{(b). } u^0 &=20x^6+y^6, &&\quad f=18000x^4y^4, &&\quad
g=20x^6+y^6.
\end{alignat*}

To see which relation gives the sought-after convergence rate, we
compare the data with a function, $y=\beta x^\alpha$, where
$\alpha=1$ in the $L^2$-case, $\alpha=\frac12$ in the $H^1$-case, and
$\alpha=\frac14$ in the $H^2$-case.  The constant, $\beta$ is
determined using a least squares fitting algorithm based on the data.

As seen in the figures below, the best $h-\eps$ relation depends on
which norm one considers.  Figures \ref{fig1}-\ref{fig2} and
\ref{fig5}-\ref{fig6} indicate that when $h=\eps^{\frac12}$,
$\|u^0-u^\eps_h\|_\lt\approx O(\eps)$ and
$\|\sigma^0-\sigma^\eps_h\|_\lt \approx
O(\eps^{\frac14})$.  It can also be seen from Figures
$\ref{fig3}-\ref{fig4}$ that when $h=\eps$,
$\|u^0-u^\eps_h\|_{H^1}=O(\eps^{\frac12})$.
\begin{figure}[tbh]
\begin{center}
\includegraphics[angle=0,width=6.25cm,height=5cm]{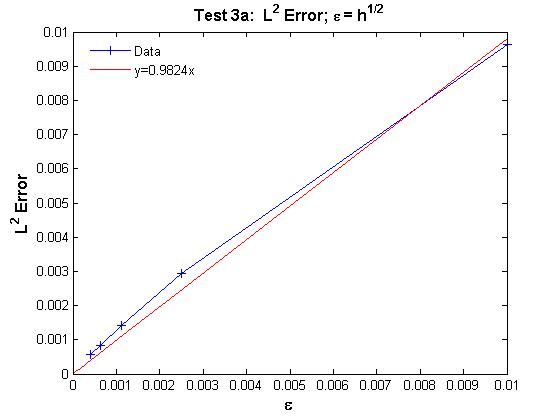}
\includegraphics[angle=0,width=6.25cm,height=5cm]{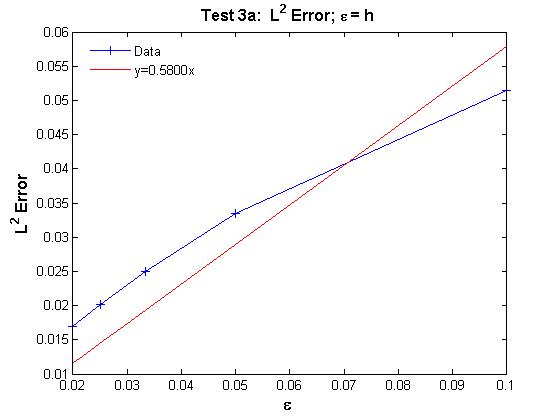}\\
\vspace{1cm}
\includegraphics[angle=0,width=6.25cm,height=5cm]{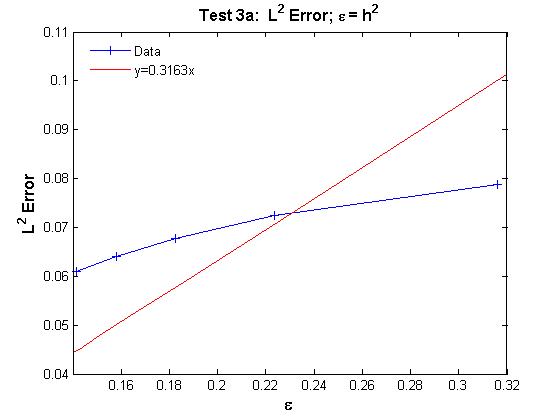}
\includegraphics[angle=0,width=6.25cm,height=5cm]{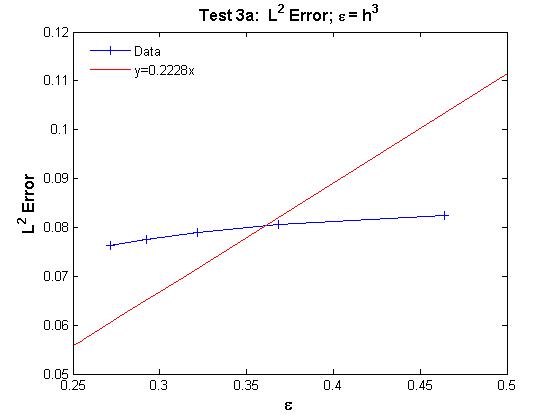}
\end{center}
\caption{{\scriptsize Test 3a.  $L^2$-error of $u^\vepsi_h$ }}\label{fig1}
\end{figure}

\begin{figure}[tbh]
\begin{center}
\includegraphics[angle=0,width=6.25cm,height=5cm]{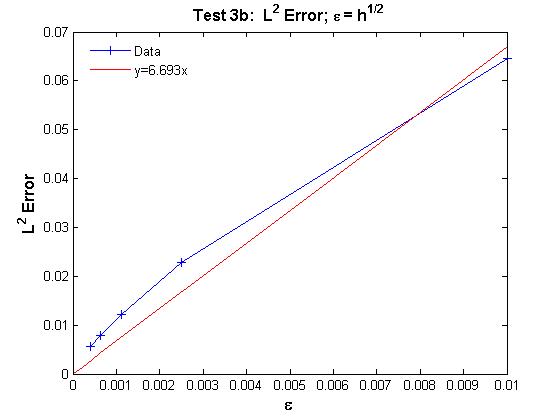}
\includegraphics[angle=0,width=6.25cm,height=5cm]{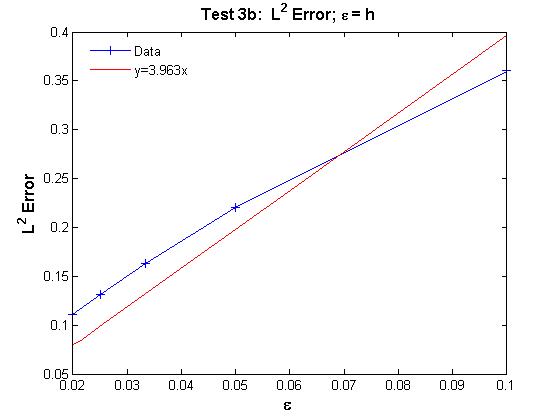}\\
\vspace{1cm}
\includegraphics[angle=0,width=6.25cm,height=5cm]{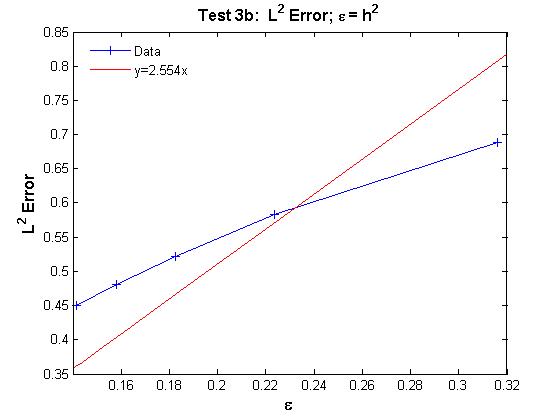}
\includegraphics[angle=0,width=6.25cm,height=5cm]{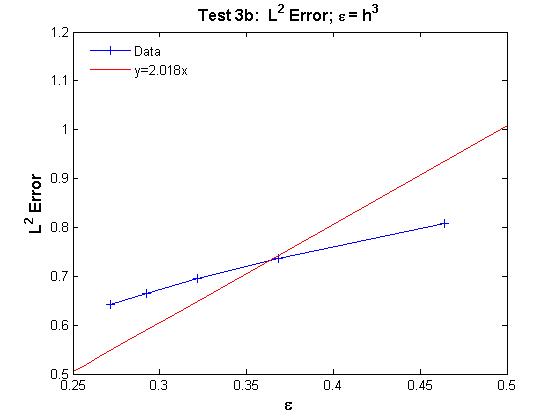}
\end{center}
\caption{{\scriptsize Test 3b. $L^2$-error of $u^\vepsi_h$ }}\label{fig2}
\end{figure}

\begin{figure}[tbh]
\begin{center}
\includegraphics[angle=0,width=6.25cm,height=5cm]{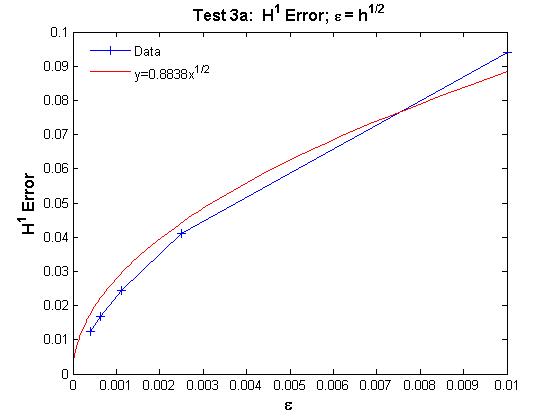}
\includegraphics[angle=0,width=6.25cm,height=5cm]{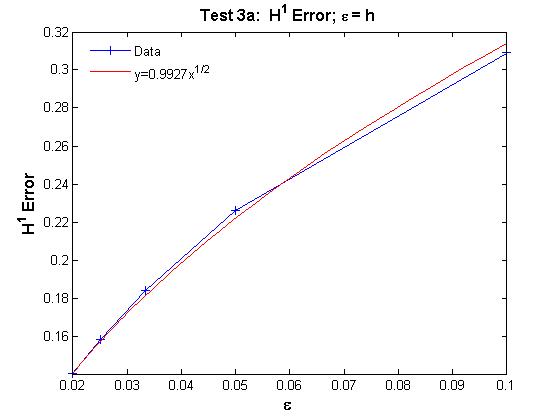}\\
\vspace{1cm}
\includegraphics[angle=0,width=6.25cm,height=5cm]{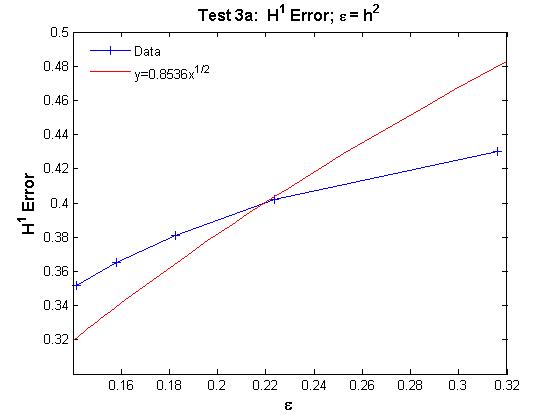}
\includegraphics[angle=0,width=6.25cm,height=5cm]{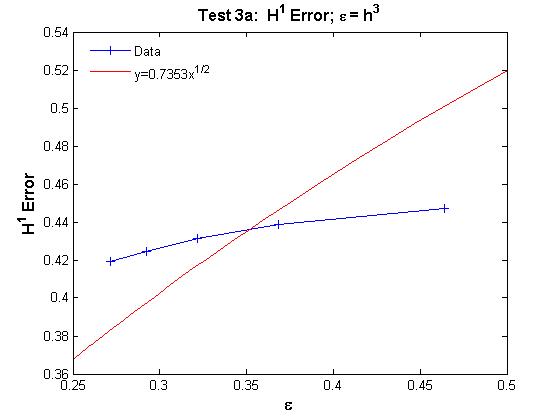}
\end{center}
\caption{\label{fig3}{\scriptsize Test 3a. $H^1$-error of $u^\vepsi_h$ }}
\end{figure}

\begin{figure}[tbh]
\begin{center}
\includegraphics[angle=0,width=6.25cm,height=5cm]{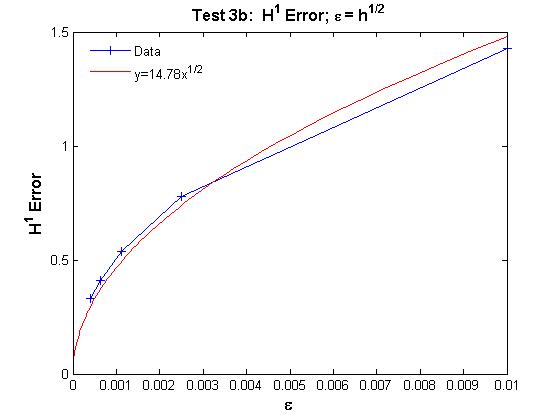}
\includegraphics[angle=0,width=6.25cm,height=5cm]{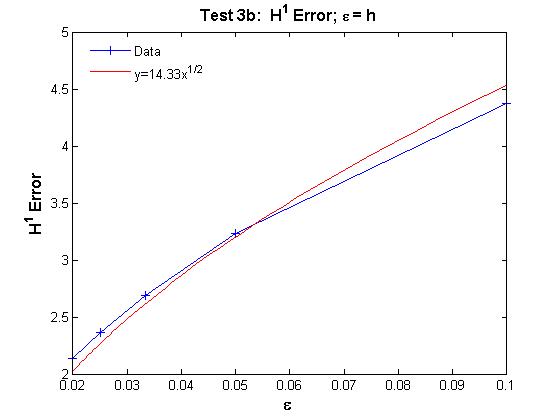}\\
\includegraphics[angle=0,width=6.25cm,height=5cm]{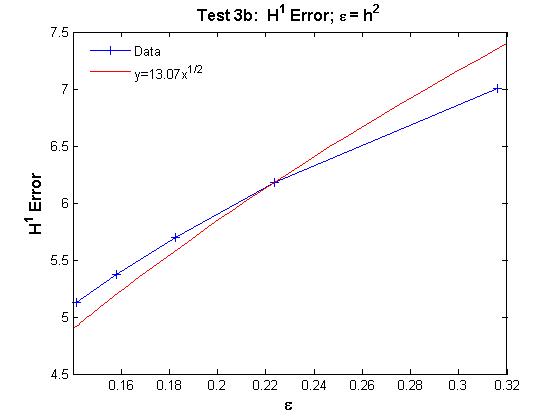}
\includegraphics[angle=0,width=6.25cm,height=5cm]{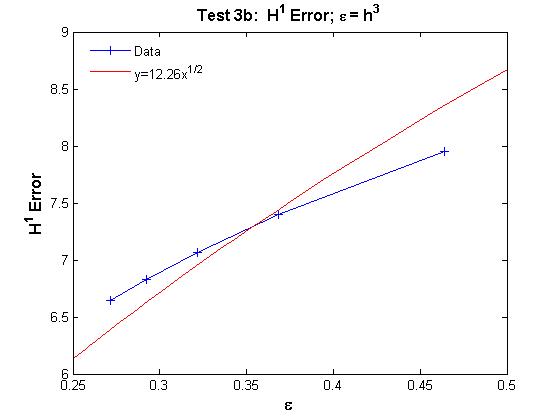}
\end{center}
\caption{{\scriptsize Test 3b. $H^1$-error of $u^\vepsi_h$}} \label{fig4}
\end{figure}

\begin{figure}[tbh]
\begin{center}
\includegraphics[angle=0,width=6.25cm,height=5cm]{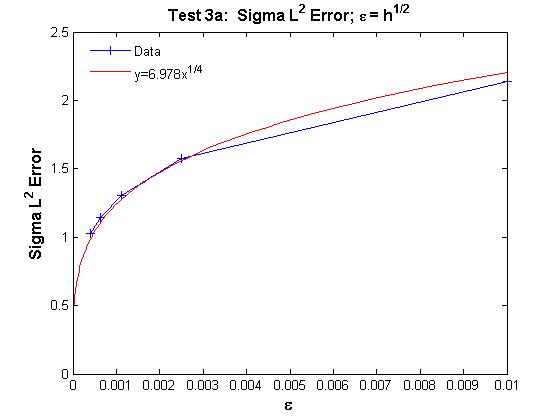}
\includegraphics[angle=0,width=6.25cm,height=5cm]{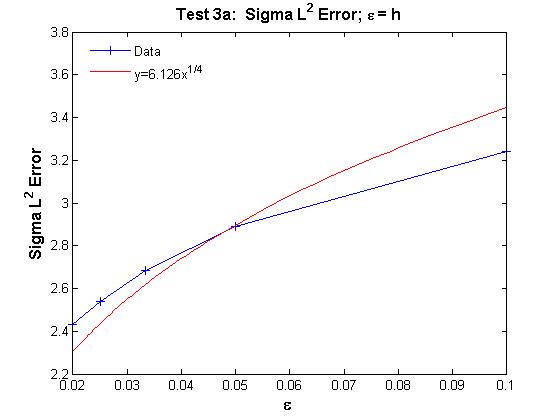}\\
\vspace{1cm}
\includegraphics[angle=0,width=6.25cm,height=5cm]{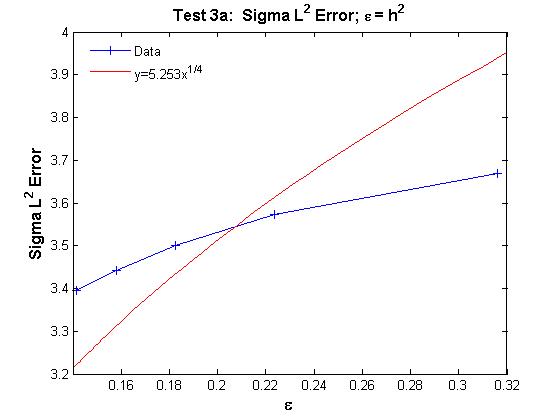}
\includegraphics[angle=0,width=6.25cm,height=5cm]{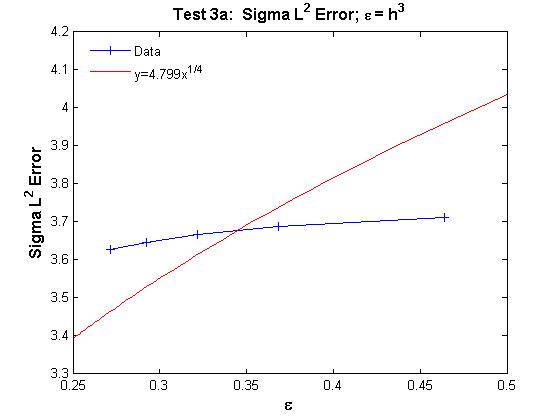}
\end{center}
\caption{{\scriptsize Test 3a. $L^2$-error of $\sigma^\vepsi_h$}}
\label{fig5}
\end{figure}

\begin{figure}[tbh]
\begin{center}
\includegraphics[angle=0,width=6.25cm,height=5cm]{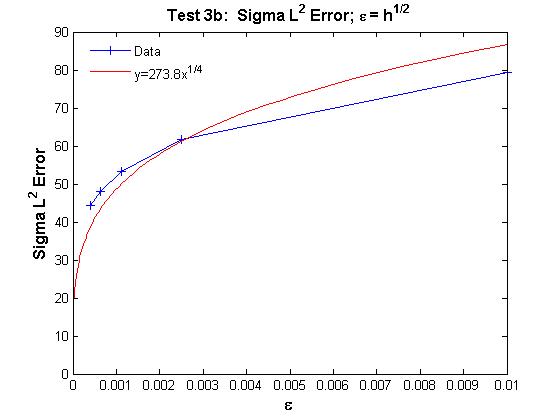}
\includegraphics[angle=0,width=6.25cm,height=5cm]{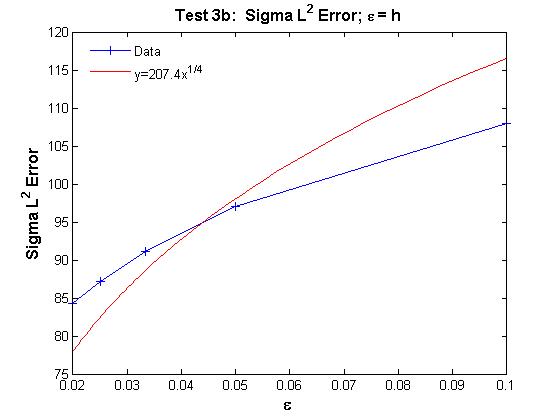}\\
\vspace{1cm}
\includegraphics[angle=0,width=6.25cm,height=5cm]{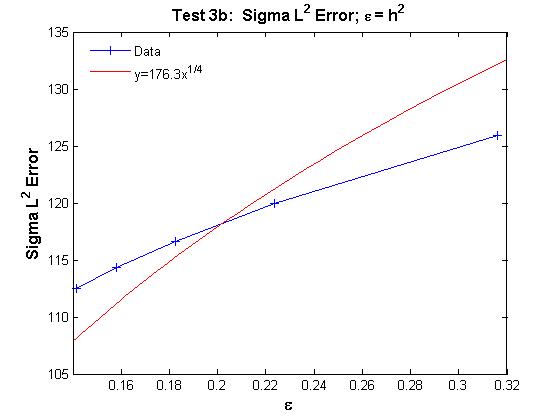}
\includegraphics[angle=0,width=6.25cm,height=5cm]{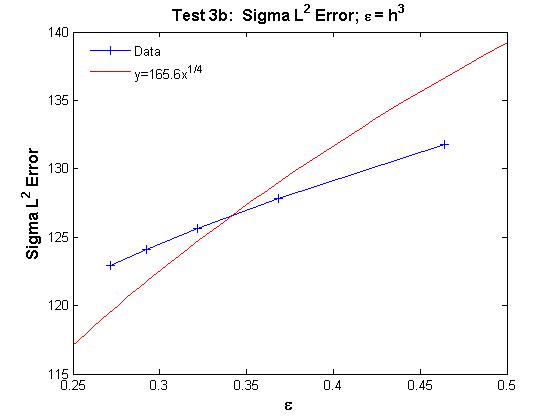}
\end{center}
\caption{{\scriptsize Test 3b. $L^2$-error of $\sigma^\vepsi_h$}}
\label{fig6}
\end{figure}
ÿ


\end{document}